\documentclass[11pt]{amsart}
\usepackage[mathscr]{eucal}
\usepackage{palatino, mathpazo, amsfonts, mathrsfs, amscd}
\usepackage[all]{xy}
\usepackage{url}
\usepackage{amssymb, amsmath, amsthm}
\usepackage{stackrel}

\newtheorem{theorem}{Theorem}[section]
\newtheorem{lemma}[theorem]{Lemma}
\newtheorem{claim}[theorem]{Claim}
\newtheorem{proposition}[theorem]{Proposition}
\newtheorem{corollary}[theorem]{Corollary}
\theoremstyle{remark}
\newtheorem{remark}[theorem]{Remark}

\newtheorem{definition}[theorem]{Definition}

\newtheorem{conventions}[theorem]{Conventions}
\numberwithin{equation}{subsection}

\newcommand{\on}{\operatorname}
\newcommand{\spec}{\operatorname{Spec}}
\newcommand{\ol}{\overline}

\newcommand{\cW}{\mathcal{W}}
\newcommand{\cY}{\mathcal{Y}}
\newcommand{\cC}{\mathcal{C}}
\newcommand{\cX}{\mathcal{X}}
\newcommand{\cH}{\mathcal{H}}

\newcommand{\CC}{\mathbb{C}}
\newcommand{\ZZ}{\mathbb{Z}}
\newcommand{\PP}{\mathbb{P}}
\newcommand{\QQ}{\mathbb{Q}}

\newcommand{\sO}{\mathscr{O}}
\newcommand{\sF}{\mathscr{F}}
\newcommand{\sH}{\mathscr{H}}

\newcommand{\Mbar}{\overline{M}}
\newcommand{\sMbar}{\overline{\mathscr{M}}}
\newcommand{\bt}{\mathbf{t}}
\newcommand{\ii}{\mathbb{1}}

\DeclareMathOperator{\Res}{Res}

\title[A Mirror Theorem for the Mirror Quintic]
{A Mirror Theorem for the Mirror Quintic}

\author{Y.-P.~Lee}
\address{Department of Mathematics, University of Utah,
Salt Lake City, Utah 84112-0090, U.S.A.}
\email{yplee@math.utah.edu}

\author{M.~Shoemaker}
\address{Department of Mathematics, University of Michigan,
Ann Arbor, MI 48109-1043, U.S.A.}
\email{shoemama@umich.edu}

\begin{document}

\begin{abstract}
%
The celebrated \emph{Mirror Theorem} states that the genus zero part of 
the $A$ model (quantum cohomology, rational curves counting) 
of the Fermat quintic threefold is equivalent to the $B$ model 
(complex deformation, variation of Hodge structure) 
of its mirror dual orbifold.
In this article, we establish a mirror-dual statement.
Namely, the \emph{$B$ model} of the Fermat quintic threefold
is shown to be equivalent to the \emph{$A$ model} of its mirror,
and hence establishes the mirror symmetry as a true duality.
\end{abstract}

\maketitle

\small
\setcounter{tocdepth}{1}
\tableofcontents
\normalsize

\setcounter{section}{-1}

\section{Introduction} \label{s:0}

\subsection{Mirror Theorem for the Fermat quintic threefold} \label{s:0.1}

Let $M$ be the Fermat quintic threefold defined by 
\[
  M := \{ x_0^5 + x_1^5 + x_3^5 + x_4^5 + x_5^5 = 0 \} \subset \PP^4.
\]
The Greene--Plesser \cite{bGrP} \emph{mirror construction} gives the mirror
\emph{orbifold} as the quotient stack
\[
  \cW := [M / \bar{G}],
\]
where $\bar{G} \cong (\ZZ/5\ZZ)^3$ is a (finite abelian) subgroup of 
the big torus of $\PP^4$ acting via generators $e_1, e_2, e_3$:
\[
 \begin{split}
  e_1[x_0, x_1, x_2, x_3, x_4] &= 
  [\zeta x_0, x_1, x_2, x_3, \zeta^{-1} x_4] \\ 
  e_2[x_0, x_1, x_2, x_3, x_4] &= 
  [ x_0, \zeta x_1, x_2, x_3, \zeta^{-1} x_4] \\ 
  e_3[x_0, x_1, x_2, x_3, x_4] &= 
  [ x_0, x_1, \zeta x_2, x_3, \zeta^{-1} x_4] . 
 \end{split}
\]
Assuming the validity of mirror symmetry for the mirror pair $(M, \cW)$, 
Candelas--de la Ossa--Green--Parkes made the celebrated calculation which
in particular predicted the number of rational curves in the Fermat
quintic of any degree.
This calculation was verified in full generality only after many years of works,
involving many distinguished mathematicians and culminating in the proof 
by A.~Givental \cite{aG1} (and Liu--Lian--Yau \cite{LLY}).
The mathematical proof of the CDGP Conjecture was termed 
the \emph{Mirror Theorem} for the Fermat quintic threefold.

In a way, what the Mirror Theorem says is that the invariants from
the complex deformations of $\cW$ matches those from 
the K\"ahler deformations of $M$, up to a change of variables termed
the \emph{mirror map}.
In terms of E.~Witten's terminology \cite{eW},
the above mirror theorem states that the (genus 0) $A$ model of $M$ is 
equivalent to $B$ model of $\cW$.
This can be formulated in mathematical terms as saying that
the genus zero Gromov--Witten theory (GWT), or quantum cohomology, on $M$
is equal to the variation of Hodge structures (VHS) associated to the
complex deformations of $\cW$.

The complex deformation of Calabi--Yau's is unobstructed by 
Bogomolov--Tian--Todorov.
The dimension of the Kodaira--Spencer space can be identified as the Hodge
number $h^{2,1}$ due to the Calabi--Yau property $K \cong \sO$.
In this case $h^{2,1}(\cW)=1$.
CDGP chose the following one-dimensional deformation family $\{ \cW_{\psi} \} = \{Q_\psi(x) = 0\}$, where
\begin{equation} \label{e:0.1}
 Q_\psi(x) = x_0^5 + x_1^5 + x_3^5 + x_4^5 + x_5^5 
  - \psi x_0 x_1 x_2 x_3 x_4 x_5
\end{equation}
of hypersurfaces in $[\PP^4/ \bar{G}]$,
such that $\psi = \infty$ is the maximally degenerate moduli point.
We note that it is often convenient to use $t = -5 \log \psi$ as the variable.
By local Torelli for Calabi--Yau, the deformation is embedded into VHS,
which then gives all information about the complex deformation.

The K\"ahler deformation is given by genus zero GWT along the ``small''
variable $t$, which is the dual coordinate for the hyperplane class $H$.
$H^{1,1}(M)_{\CC}$ is often called the complexified K\"ahler moduli.

We can rephrase the above in much more precise terms.
Both genus zero GWT and VHS can be described by
differential systems associated to flat connections. 
For GWT, it is the \emph{Dubrovin connection}; 
for VHS the \emph{Gauss--Manin} connection.
The definitions can be found in Sections~\ref{s:1} and \ref{s:4} respectively.
Therefore, we can phrase the Mirror Theorem for the Fermat quintic in the 
following form.

\begin{theorem}[$=$ Theorem~\ref{t:MTfull}] \label{t:0.1}
The fundamental solutions of the Gauss--Manin connection for $\cW_t$
are equivalent, up to a mirror map, to the fundamental solutions of the 
Dubrovin connection for $M$, when restricted to $H^2(M)$.
\end{theorem}

\subsection{Mirror Theorem for the mirror quintic} \label{s:0.2}

Theorem~\ref{t:0.1} can be stated suggestively as
\[
 \text{$A$ model of $M$ $\equiv$ $B$ model of $\cW$.}
\]
In order for the mirror symmetry to be a true duality, one will also
have to show that
\[
 \text{$B$ model of $M$ $\equiv$ $A$ model of $\cW$.}
\]
This is the task we set for ourselves in this paper.

The first thing we note is that $\cW$ is an orbifold.
Thus we must replace the singular cohomology by the Chen--Ruan cohomology,
and the usual Gromov--Witten theory by the orbifold GWT.
These are defined in Section~\ref{s:1}.

Upon a closer look, however, there is a serious technical issue.  In the B
model of $M$, the Kodaira--Spencer space 
is of dimension $101$ and 
the VHS of $H^3(M)$ is a system of rank $204$, thus a calculation of
the full Gauss--Manin connection for $M$ is unfeasible.
As a first step however, we choose a one-dimensional deformation family $\{ M_t \}$
defined by the vanishing of \eqref{e:0.1}, 
\emph{reinterpreted as a family in $\PP^4$}.
Similarly, in the $A$ model of $\cW$, we have 
$h^{1,1}_{CR}(\cW) =101$, where the subscript denotes 
Chen--Ruan cohomology.
We choose the one-dimensional subspace of the complexified K\"ahler moduli
spanned by the hyperplane class and call the coordinate $t$ as before.
These one dimensional families are arguably 
\emph{the most natural and the most important dimension}.

With these choices, the Gauss--Manin system for $M$ 
still has rank $204$, but over a one dimensional base.
The fundamental solution is a matrix of size $204$ by $204$ in one variable.
The Dubrovin connection on $H^{even}_{CR}(\cW)$ likewise has the
fundamental solution matrix of size $204$ by $204$.
Here $204 = \dim H^{even}_{CR}(\cW)$.

The main result of this paper is the following theorem.

\begin{theorem}[$=$ Theorem~\ref{t:MMTfull}] \label{t:0.2}
The fundamental solutions of the Gauss--Manin connection for $\{ M_t \}$ are
equivalent, up to a mirror map, to the fundamental solutions of
the Dubrovin connection for $\cW$ restricted to $t \in H^2(\cW)$.
\end{theorem}

\subsection{Outline of the paper} \label{s:0.3}
We have in mind the readership with diverse background.
For convenience, we have included short introductions in Section~1
and Section~4 to orbifold Gromov--Witten theory and 
the theory of variation of Hodge structures, 
recalling only facts pertinent to our presentation.
Sections~2 and 3 present the $A$ model calculation for $\cW$.
We first calculate the genus zero Gromov--Witten theory for $[\PP^4/\bar{G}]$
in Section~\ref{s:2}; 
we then calculate the genus zero Gromov--Witten theory for $\cW$ in 
Section~\ref{s:3}.
In Section~5 we present a reformulation of the results from \cite{DGJ},
and summarize our $B$ model calculation for $M_t$.
In the last section, we prove our main result, showing
the validity of the Mirror-dual statement of the Mirror Theorem.
For the benefit of our dual readership, we include a derivation
of Theorem~\ref{t:0.1} from the usual statement of the Mirror Theorem.

\subsection*{Acknowledgements}
Y.P.L.\ would like to thank his collaborators Profs.~H.-W.~Lin and C.-L.~Wang.
In particular, he learns most of what little he knows about the Hodge theory 
from his collaborative projects with them.
Y.P.L.\ is partially supported by the NSF.

M.S.\ would like to thank his advisor, Prof.~Y.~Ruan for his
help and guidance over the years, and for first introducing him to 
this beautiful subject.  
He is also grateful to Prof.~R.~Cavalieri for many useful conversations.
M.S.\ was partially supported by NSF RTG grant DMS-0602191.

\section{Quantum orbifold cohomology} \label{s:1}
In this section we give a brief review of Chen--Ruan cohomology and
quantum orbifold cohomology, 
with the parallel goal of setting notation.
A more detailed general review can be found in \cite{CCLT}.

\begin{conventions} \label{conv:1}
We work in the algebraic category.
The term \emph{orbifold} means ``smooth separated Deligne--Mumford stack
of finite type over $\mathbb{C}$.''

The various dimensions are complex dimensions.
On the other hand, the degrees of cohomology are all in real/topological 
degrees.
  
Unless otherwise stated all cohomology groups have coefficients
in $\CC$.
\end{conventions}

\subsection{Chen--Ruan cohomology groups} \label{s:1.1}

Let $\cX$ be a stack. 
Its inertia stack $I\cX$ is the fiber product
\[
\xymatrix{
I\cX \ar[r] \ar[d] & \cX \ar[d]^\Delta \\
\cX\ar[r]^\Delta & \cX \times \cX \\ 
}
\]
where $\Delta$ is the diagonal map.  
The fiber product is taken in the $2$-category of stacks. 
One can think of a point of $I\cX$ as
a pair $(x,g)$ where $x$ is a point of $\cX$ and $g \in \on{Aut}_{\cX}(x)$.  
There is an involution $I: I\cX \to I\cX$ 
which sends the point $(x,g)$ to $(x,g^{-1})$.
It is often convenient to call the components of $I\cX$ for which
$g \neq e$ the \emph{twisted sectors}.

If $\cX = [V/G]$ is a global quotient of a nonsingular
variety $V$ by a finite group $G$, $I\cX$ takes a particularly simple form.
Let $S_G$ denote the set of conjugacy classes $(g)$ in $G$,
then
\[
 I [V/G] = \coprod_{(g) \in S_G} [ V^g/C(g) ].
\]

The \emph{Chen--Ruan orbifold cohomology groups} $H^*_{CR}(X)$ (\cite{CR1})
of a Deligne--Mumford stack $\cX$ are the cohomology groups of 
its inertia stack
\[
  H_{CR}^*(\cX) := H^*(I\cX).
\]

Let $(x, g)$ be a geometric point in a component $\cX_i$ of $I\cX$. 
By definition $g \in \on{Aut}_{\cX}(x)$.
Let $r$ be the order of $g$.
Then the $g$-action on $T_x \cX$ decomposes as eigenspaces
\[
T_x \cX = \bigoplus_{0 \leq j < r} E_j
\]
where $E_j$ is the subspace of $T_x \cX$ on which $g$ acts by
multiplication by $\exp(2\pi\sqrt{-1}j/r)$. 
Define the age of $\cX_i$ to be
\[
 \on{age}(\cX_i) := \sum_{j=0}^{r-1} \frac{j}{r} \dim( E_j).
\]
This is independent of the choice of geometric point $(x,g) \in \cX_i$. 

Let $\alpha$ be an element in $H^p(\cX_i) \subset H^*(I\cX)$.
Define the age-shifted degree of $\alpha$ to be
\[
 \deg_{CR}(\alpha) := p + 2 \on{age}(\cX_i).
\]
This defines a grading on $H_{CR}(\cX)$.

When $\cX$ is compact the {\em orbifold Poincar\'e pairing} is defined by
\[
 (\alpha_1, \alpha_2)^{\cX}_{CR} := \int_{I\cX} \alpha_1 \cup I^*(\alpha_2),
\]
where $\alpha_1$ and $\alpha_2$ are elements of $H^*_{CR}(\cX)$.
It is easy to see that when $\alpha_1$ and $\alpha_2$ are homogeneous elements,
$(\alpha_1, \alpha_2)_{CR} \neq 0$ only if 
$\deg_{CR} (\alpha_1) + \deg_{CR} (\alpha_2) = 2 \dim (\cX)$.

\subsection{Orbifold Gromov-Witten theory} \label{s:1.2}

\subsubsection{Orbifold Gromov--Witten invariants}

We follow the standard references \cite{CR2} and \cite{AGV} 
of orbifold Gromov--Witten theory.

Given an orbifold $\cX$, there exists a moduli space $\sMbar_{g,n}(\cX, d)$ 
of stable maps from $n$-marked genus $g$ pre-stable orbifold curves 
to $\cX$ of degree $d \in H_2(\cX; \QQ)$.
Each source curve $(\cC, p_1, \ldots, p_n)$ has non-trivial orbifold structure 
only at the nodes and marked points:
At each (orbifold) marked point it is a cyclic quotient stack and
at each node a \emph{balanced} cyclic quotient. That is, 
\'etale locally isomorphic to 
\[ 
  \left[ \spec \left( \frac{\CC[x, y]}{(xy)} \right) / \mu_r \right],\] 
where $\zeta \in \mu_r$ acts as $(x,y) \mapsto (\zeta x, \zeta^{-1} y)$.  
The maps are required to be representable at each node.

Each marked point $p_i$ is \'etale locally isomorphic to $[\CC/ \mu_{r_i}]$.  
There is an induced homomorphism 
\[ 
  \mu_{r_i} \to \on{Aut}_{\cX}(f(p_i)).
\]  
Maps in $\sMbar_{g,n}(\cX, d)$ are required be representable, which
amounts to saying that these homomorphisms be injective.  
For each marked point $p_i$, one can thus associate a point $(x_i, g_i)$ in 
$I\cX$ where $x_i = f(p_i)$, and $g_i \in \on{Aut}_{\cX}(x_i)$ is the 
image of $\exp(2 \pi \sqrt{-1}/r_i)$ under the induced homomorphism.  

Given a family $\cC \to S$ of marked orbifold curves, there may be nontrivial 
gerbe structure above the locus defined by the $i$-th marked point.  
For this reason there is generally not a well defined map 
\[ 
  ev_i: \sMbar_{g,n}(\cX, d) \to I\cX.
\] 
However, as explained in \cite{AGV} and \cite{CCLT} Section~2.2.2,  
it is still possible to define maps 
\[ 
  ev_i^*: H^*_{CR}(\cX) \to H^*(\sMbar_{g,n}(\cX, d))
\] 
which behave \emph{as if the evaluation maps $ev_i$ are well defined}.

Let $X$ denote the coarse underlying space of the stack $\cX$.  
There is a \emph{reification map} 
\[
  \sMbar_{g,n}(\cX, d) \to \sMbar_{g,n}(X, d),
\] 
which forgets the orbifold structure of each map.  
For each marked point there is an associated 
line bundle, the $i^{th}$ universal cotangent line bundle,
\[ 
 \begin{array}{c} L_i \\ \downarrow \\ \sMbar_{g,n}(X, d) \end{array} 
\] 
with fiber $T^*_{p_i} C$ over $\{f: (C, p_1, \ldots, p_n) \to X\}$.  
Define the $i$-th $\psi$-class
 by  $\psi_i = r^*(c_1(L_i))$.  

As in the non-orbifold setting, there exists a virtual fundamental class 
$[\sMbar_{g,n}(\cX, d)]^{vir}$.
\emph{Orbifold Gromov-Witten invariants} for $\cX$ are defined as integrals
\begin{equation*} 
\big\langle \alpha_1 \psi^{k_1}, \ldots , \alpha_n \psi^{k_n}\big\rangle_{g,n,d}^{\cX}  = 
\int_{[\sMbar_{g,n}(\cX, d)]^{vir}} \prod_{i=1}^n 
ev_i^*(\alpha_i) \psi_i^{k_i}, 
\end{equation*}
where $\alpha_i \in H^*_{CR}(\cX)$.

Let $\sMbar_{g,(g_1, \ldots, g_n)}(\cX, d)$ denote the open 
and closed substack of $\sMbar_{g,n}(\cX, d)$ such that $ev_i$ maps to a 
component $\cX_{g_i}$ of $I\cX$.  
The space $\sMbar_{g,(g_1, \ldots, g_n)}(\cX, d)$ has (complex)
virtual dimension 
\[
  n + (g - 1)(\dim \cX - 3) +  \langle c_1(T\cX),d \rangle 
  -  \sum_{i=0}^n \on{age}(\cX_{g_i}).
\]  
In other words, for homogeneous classes $\alpha_i \in H^*(\cX_{g_i})$ 
the Gromov-Witten invariant \\
$\big\langle \alpha_1, \ldots , \alpha_n\big\rangle_{g,n,d}^{\cX}$ 
will vanish unless 
\[ 
  \sum_{i = 1}^n \deg_{CR}(\alpha_i) = 2\left(n + (g - 1)(\dim \cX - 3) 
  +  \langle c_1(T\cX),d \rangle\right).
\]

\subsubsection{Quantum cohomology and the Dubrovin connection} \label{s:1.2.5}

Let $\{T_i\}_{i \in I}$ be a basis for $H^*_{CR}(\cX)$ and 
$\{ T^i \}_{i \in I}$ its dual basis.
We can represent a general point in coordinates by
\[
 \bt = \sum_i t^i T_i   \in H^*_{CR}(\cX).
\]
Gromov-Witten invariants allow us to define a family of product structures 
parameterized by $\bt$ in a formal neighborhood of $0$ in $H^*_{CR}(\cX)$.  
The \emph{(big) quantum product} $*_\bt$ is defined as
\begin{equation}\label{e:product}
 \alpha_1 *_\bt \alpha_2 :=  \sum_d \sum_{n \geq 0} \sum_i
 \frac{q^d}{n!} \langle \alpha_1, \alpha_2, T_i,  
 \bt, \ldots, \bt \rangle^\cX_{0, 3+n, d} T^i ,
\end{equation}
where the first sum is over the Mori cone of effective curve classes 
and the variables $q^d$ are in an appropriate Novikov ring $\Lambda$ used to 
guarantee formal convergence of the sum.  
The \emph{WDVV equations} (\cite{CK}, Section~8.2.3) imply the associativity
of the product.
The \emph{small quantum product} is defined by restricting the parameter of
the quantum product to divisors $\bt \in H^2(\cX)$ supported on the 
\emph{non-twisted sector}. 

One can interpret $*_\bt$ as defining a product structure on the tangent 
bundle $T H^*_{CR}(\cX; \Lambda)$, such that for a fixed $\bt$ 
the quantum product defines a (Frobenius) algebra structure on 
$T_{\bt} H^*_{CR}(\cX; \Lambda)$.
This can be rephrased in terms of the \emph{Dubrovin connection}:
\[ 
  \nabla^z_{\frac{\partial}{\partial t^i}} \left( \sum_{j} a_j T_j\right) = 
 \sum_{j} \frac{\partial a_j} {\partial t^i} T_j 
 - \frac{1}{z} \sum_{j} a_j T_i *_\bt T_j .
\]  
This defines a $z$-family of connections on $T H^*_{CR}(\cX; \Lambda)$.

\begin{remark} 
Note that when $\bt$, $T_i$ and $T_j$ are in $H^{even}_{CR}(\cX)$, then 
for dimension reasons $T_i *_\bt T_j$ will be also be supported in even degree.
Thus $\nabla^z$ restricts to a connection on $TH^{even}_{CR}(\cX; \Lambda)$.
When restricted to $TH^{even}_{CR}(\cX; \Lambda)$, the quantum product is commutative.
\end{remark}

\begin{remark} \label{r:1.3}
For the purpose of this paper, we clarify here what we mean by 
``$A$ model of $\cX$''.
Let $H := H^{even}_{CR}(\cX; \Lambda)$.
The (genus zero part of) \emph{$A$ model of $\cX$} is the tangent bundle 
$TH$ with its natural (flat) fiberwise pairing and the Dubrovin connection 
restricted to $H^{1,1}_{CR}(\cX)$.
\end{remark}

The commutativity and associativity of the quantum product implies that the 
Dubrovin connection is flat.
The \emph{topological recursion relations} 
allow us to explicitly describe solutions to $\nabla^z$.  
Define 
\begin{equation}\label{e:bsol}
 s_i(\bt,z) = T_i + \sum_d \sum_{n \geq 0} \sum_j \frac{q^d}{n!} 
 \bigg\langle \frac{T_i}{z - \psi_1}, T^j,  
 \bt, \ldots, \bt \bigg\rangle^\cX_{0, 2+n, d} T_j
\end{equation}  
where ${1}/{(z - \psi_1)}$ should be viewed as a power series in $1/z$.  
The sections $s_i$ form a basis for the $\nabla^z$-flat sections;
see e.g.~\cite{CK}, Proposition 10.2.1.
Thus we obtain a fundamental solution matrix $S = S(\bt, z) = (s_{ij})$ given by 
\begin{equation}\label{e:sol}
 s_{ij}(\bt,z)= ( T^i, s_j )_{CR}^\cX. 
\end{equation}

If one restricts the base to divisors $\bt \in H^2(\cX)$, 
the \emph{divisor equation} (\cite{AGV} Theorem~8.3.1)
allows a substantial simplification of the formula for $s_i$
\[
 s_i(\bt,z)|_{\bt \in H^2(\cX)} = 
 e^{\bt / z}\left(T_i + \sum_{d > 0} \sum_j
 q^d e^{d\bt} \bigg\langle \frac{ T_i}{z - \psi_1}, 
 T^j \bigg\rangle^\cX_{0, 2,d}  T_j\right).
\]

\subsection{Generating functions} \label{s:1.3}

Given an orbifold $\cX$, Givental's (big) $J$-function is the first row vector
of the fundamental solution matrix, obtained by pairing
the solution vectors of the Dubrovin connection with $1$.
\begin{equation*} 
 \begin{split}
   J^\cX_{big}(\bt, z) & :=  \sum_{i} \left( s_i(\bt), 1 \right)_{CR}^\cX T^i \\
   & =  1 + \sum_d \sum_{n \geq 0} \sum_i \frac{q^d}{n!}
\bigg\langle \frac{T_i}{z - \psi_1}, 1,  
\bt, \ldots, \bt \bigg\rangle^\cX_{0, 2+n, d} T^i \\
  & = 1 + \frac{\bt}{z} + \sum_{d} 
   \sum_{n \geq 0} \sum_i \frac{q^d}{n!} \bigg\langle 
  \frac{T_i}{z(z - \psi_1)}, \bt, \ldots, \bt \bigg\rangle^\cX_{0,1+n, d} T^i ,
 \end{split}
\end{equation*}
The last equality follows from the \emph{string equation}.
It is also easy to see that the fundamental solution matrix
$S(\bt, z)$ of \eqref{e:sol} is equal to $z \nabla J_{big}$.
As such, $J_{big}$ encodes all information about quantum cohomology.

However, the big $J$-function is often impossible to calculate directly.
In the non-orbifold Gromov--Witten theory, when the cohomology is 
generated by divisors, the \emph{small $J$-function} proves much more
computable, while powerful enough to solve many problems; 
see e.g.~\cite{aG1, aG2}.
The small $J$-function for a nonsingular \emph{variety} $X$ is
a function on $\bt \in H^2(X)$:
\[ \begin{split}
 J^X_{small} (\bt,z) &:= J^X_{big} (\bt,z) |_{\bt \in H^2(X)} \\
 & = e^{\bt / z}\left(1 + \sum_{d > 0} \sum_{i}
q^d e^{d\bt} \bigg\langle \frac{ T_i}{z - \psi_1}, 
1 \bigg\rangle^X_{0, 2,d}  T^i\right).
 \end{split}
\]

In orbifold theory, however, the Chen--Ruan cohomology is never generated by 
divisors except for trivial cases, due to the presence of the twisted sectors.
Therefore, the knowledge of the small $J$-function alone is often not enough to 
reconstruct significant information about the orbifold quantum cohomology.
(Note however that in Section~5 of \cite{CCLT} one way was found to 
circumvent this obstacle for weighted projective spaces.)

We propose the following definition of 
\emph{small $J$-matrix for orbifolds}.

\begin{definition} \label{d:sJ}
For $\bt \in H^2(\cX)$, define $J_g^\cX$ as the cohomology-valued function
\begin{equation}\label{e:sJ}
 \begin{split}
 J^{\cX}_{g} (\bt, z)|_{\bt \in H^2(\cX)} 
  & := \sum_i \left(s_i(\bt)|_{\bt \in H^2(X)}, \ii_g\right)_{CR}^\cX T^i \\
  &= e^{\bt/z}\left(\ii_g + \sum_{d > 0} \sum_i q^d e^{d\bt}  \bigg\langle 
   \frac{T_i}{z - \psi_1}, \ii_g \bigg\rangle^{\cX}_{0,2,d} T^i\right),
 \end{split}
\end{equation}
where $\ii_g$ is the fundamental class on the component $\cX_g$ of $I\cX$.

The \emph{small $J$-matrix} is the matrix-valued function
\[
 J^{\cX}_{small} (\bt, z) =  
 \left[ J^{\cX}_{g,i} (\bt, z) \right]_{g \in G, i \in I}
 =  \left[ (J^{\cX}_{g} (\bt, z), T_i)^{\cX}_{CR} \right]_{g \in G, i \in I}\; ,
\]
where $G$ is the index set of the components of $I\cX$,
$I$ the index for the basis $\{T_i\}_{i \in I}$ of $H^*_{CR}(\cX)$ and
$J^\cX_{g,i}(\bt, z)$ the coefficient of $T^i$ in $J^\cX_g(\bt,z)$.
\end{definition}

\begin{remark} \label{r:1.5}
We believe that the small $J$-matrix is the right replacement of the small 
$J$-function in the orbifold theory, 
for its computability and structural relevance.

Structurally equation~\eqref{e:sol} shows that one needs to specify 
``two-points'' (i.e.~a matrix) in the generating function in order to form 
the fundamental solutions of the Dubrovin connection.
Ideally, one would like to get the full $|I| \times |I|$ fundamental
solution matrix $S = z \nabla J_{big}$ restricted to $\bt \in H^2(\cX)$.
This would give all information about the \emph{small} quantum cohomology.
Unfortunately, a direct computation of $S(\bt)|_{\bt \in H^2(\cX)}$
is mostly out of reach in the orbifold theory.

In the (non-orbifold) case when $H^*(X)$ is generated by divisors,
as shown by A.~Givental, the small $J$-function is often enough to determine 
the essential information for small quantum cohomology.
One can think of the small $J$-function as a
a submatrix of size $1 \times |I|$, indeed the first row vector, of $S$. 

However, in the orbifold theory, the above matrix is not
enough to determine useful information about small quantum
cohomology except in the trivial cases.
We believe that the smallest useful submatrix of $S$ 
is the small $J$-matrix (of size $|G| \times |I|$) defined above.
We will show that it is both computable and relevant to the
structure of orbifold quantum cohomology. In this paper we are able to calculate
the small $J$-matrix of the 
toric orbifold $\cY = [\PP^4/\bar{G}]$, and we use a sub-matrix
of the small $J$-matrix $J^{\cW}_{small}$ to
fully describe the solution matrix $S(\bt)|_{\bt \in H^2(\cX)}$ 
of the mirror quintic $\cW$.
\end{remark}

\section{$J$-function of $[\PP^4/\bar{G}]$} \label{s:2}

\subsection{Inertia orbifold of $[\PP^4/\bar{G}]$} \label{s:2.1}
Let $[x_0, x_1, x_2, x_3, x_4]$ be the homogeneous coordinates of $\PP^4$.
Denote
\[
 \zeta = \zeta_5 := e^{2 \pi \sqrt{-1}/5}.
\]
Let the group $\bar{G} \cong (\ZZ/5\ZZ)^3$ be a (finite abelian) subgroup of 
the big torus of $\PP^4$ acting via generators $e_1, e_2, e_3$:
\begin{equation} \label{e:G}
 \begin{split}
  e_1[x_0, x_1, x_2, x_3, x_4] &= 
  [\zeta x_0, x_1, x_2, x_3, \zeta^{-1} x_4] \\ 
  e_2[x_0, x_1, x_2, x_3, x_4] &= 
  [ x_0, \zeta x_1, x_2, x_3, \zeta^{-1} x_4] \\ 
  e_3[x_0, x_1, x_2, x_3, x_4] &= 
  [ x_0, x_1, \zeta x_2, x_3, \zeta^{-1} x_4] . 
 \end{split}
\end{equation}
Let $\cY =[\PP^4/ \bar{G}]$.
As explained in the Introduction this orbifold plays an instrumental role 
in what follows so we give here a detailed presentation of its corresponding 
inertia orbifold.  

The group $\bar{G}$ can be described alternatively as follows. Let
\[
  {G} := \{(\zeta^{r_0}, \ldots, \zeta^{r_4}) \, | \,
  \sum_{i=0}^4 r_i \equiv 0 \, (\on{mod} 5) \}
\] 
and 
\[
 \bar{G} \cong {G}/ \big\langle (\zeta, \ldots, \zeta) \big \rangle.
\]
The $\bar{G}$-action on $\PP^4$ comes from coordinate-wise multiplication.  
By a slight abuse of notation, we will represent a group element 
$g  \in {G}$ by the power of $\zeta$ in each coordinate:
\[
 G = \{ (r_0, \ldots , r_4) \, | \, \sum_{i=0}^4 r_i \equiv 0 \, (\on{mod} 5), 
  0 \leq r_i \leq 4 \,\forall i \}.
\]
For an element $g \in G$, denote $[g]$ the corresponding element in $\bar{G}$.

Fix an element $\bar{g} \in \bar{G}$. 
Let $g = (r_0, \ldots , r_4) \in {G}$ be such that $[g] =\bar{g}$. Define
\[
  I(g) :=  \left\{ j \in \{0, 1, 2, 3, 4\} \, | \, r_j =0 \right\},
\]
then
\[
  \PP^4_g := \left\{x_j = 0\right\}_{j \notin I(g)} \subset \PP^4
\] 
is a component of $(\PP^4)^{\bar{g}}$.  
From this we see that each element $g \in {G}$ such that 
$[g] =\bar{g}$ corresponds to a connected component $\cY_{g}$ of $I\cY$ 
associated with $\PP^4_g \subset (\PP^4)^{\bar{g}}$. 
Note that if $g$ has no coordinates equal to zero then $\PP^4_g$ 
is empty, and so is $\cY_g$.
This gives us a convenient way of indexing components of $I\cY$.  

We summarize the above discussions in the following lemma.
\begin{lemma} \label{l:2.1}
\begin{equation*} 
  I\cY = \coprod_{g \in S} \cY_{g}\:,
\end{equation*} 
where 
\[
 \cY_g = \{ (x, [g]) \in I\cY \, | \, x \in [\PP^4_g / \bar{G} ]\}
\]
is a connected component and $S$ denotes the set of
all $g = (r_0, \ldots, r_4)$ 
such that at least one coordinate $r_i$ is equal to $0$.

Consequently, a convenient basis $\{T_i\}$ for $H^*_{CR}(\cY)$ is 
\[ 
 \bigcup_{g \in S}   \{\ii_g, \ii_g \tilde{H}, \ldots , 
   \ii_g\tilde{H}^{\dim(\cY_g)}\}.
\] 
\end{lemma}

\subsection{$J$-functions} \label{s:2.2}

Recalling a basic fact about global quotient orbifolds, 
a map of orbifolds $f: \cC \to [\PP^4/\bar{G}]$ can be identified with 
a principal $\bar{G}$-bundle $C$, and a $\bar{G}$-equivariant map 
$\tilde{f}: C \to \PP^4$ such that the following diagram commutes:
\footnote{Technically f is identified with an equivalence class of such 
objects.}
\begin{equation}\label{e:cover}
 \xymatrix{C \ar[d]_{\pi_{C}}\ar[r]^{\tilde{f}} & \PP^4 \ar[d]_{\pi_{\PP^4}}\\
  \cC\ar[r]^{f} & [\PP^4/\bar{G}].}
\end{equation}

\begin{lemma} \label{l:2.2}
(i) The map $\tilde{f}$ is representable if and only if 
$C$ is a nodal curve with each irreducible component a smooth
variety.

(ii) There do not exist representable orbifold morphisms $f: \cC \to \cY$ 
from a genus $0$ orbifold curve $\cC$ with only one orbifold marked point.
\end{lemma}

\begin{proof}
(i) follows immediately from the definition of representability.

(ii) follows from (i):
If $\cC$ is irreducible, this is because there do not exist smooth covers of 
genus $0$ orbifold curves with only one point with nontrivial isotropy.  
An induction argument then shows that the same is true of reducible curves 
with only one orbifold marked point (we assume always that our nodes be balanced).
\end{proof}

A line bundle on $[\PP^4/\bar{G}]$ can be identified with a $\bar{G}$-equivariant
line bundle on $\PP^4$.
Therefore, the Picard group on $[\PP^4/\bar{G}]$ is a $\bar{G}$-extension of $\ZZ$.
Let $L$ be \emph{any} line bundle on $[\PP^4/\bar{G}]$ such that
$\pi_{\PP^4}^*L = H$, where $H$ is the hyperplane class on $\PP^4$.
By \eqref{e:cover}, we have the following equality
\[
  \int_{\cC} f^*(L) = \frac{1}{125} \int_{C} \tilde{f}^*(H).
\]
We define the degree of a map $f:\cC \to \cY$ by
\[
  d := \frac{1}{125} \int_{C} \tilde{f}^*(H).
\]

This also allows us to determine necessary conditions on the triple $d$,
$h = (r_0(h), \ldots, r_4(h))$ and $g = (r_0(g), \ldots, r_4(g))$ for 
\[
  \sMbar_{0,h,g}(\cY, d) := 
  \sMbar_{0,2}(\cY, d) \cap ev_1^{-1}(\ii_h) \cap ev_2^{-1}(\ii_g)
\]
to be nonempty.  

\begin{proposition}\label{p:mapsallowed} 
The space $\sMbar_{0,h,g}(\cY, d) $ is nonempty only if 
\begin{enumerate}
\item[(i)]  $[h] = [g]^{-1}$ in $\bar{G}$;
\item[(ii)]  $r_i(h) + r_i(g) \equiv 5d \,(\on{mod} 5)$ or equivalently
${\displaystyle \langle d \rangle = \langle  (r_i(h) + r_i(g))/5 \rangle }$ for $0 \leq i \leq 4$.
\end{enumerate}
\end{proposition}

\begin{proof}  
We will first consider the case where the source curve is irreducible.
Assume that there exists a map $\{f: \cC \to \cY\}$ in 
$\sMbar_{0,h,g}(\cY, d)$ such that $\cC$ is non-nodal.  
Consider the principal $\bar{G}$-bundle $\pi_C: C \to \cC$.  
After choosing a generic base point $x \in \cC$ and a point $\tilde{x}$ in 
$\pi_C^{-1}(x)$, we get a homomorphism $\phi: \pi_1(\cC, x) \to \bar{G}$.  
We can specify generators $\rho_1$, and $\rho_2$ of $\pi_1(\cC, x)$ 
such that $\rho_i$ is the class of loops wrapping once around $p_i$ 
in the counterclockwise direction.  
Then $\phi(\rho_1) = [h]$ and $\phi(\rho_2) = [g]$.  
Because $\rho_1\cdot \rho_2 = 1$ in $\pi_1(\cC,x)$, 
it must be the case that $[h]\cdot [g] = 1$ in $\bar{G}$.
This proves (i) for $\cC$ non-nodal.

Next we will show (ii) in the case where $\cC$ is non-nodal.
To see this, note that the only smooth connected cover of $\cC$ is isomorphic 
to $\PP^1$. This cover is degree $r := |[h]|$, 
so $C$ must consist of $|\bar{G}|/r$ components, each isomorphic to $\PP^1$.  
In the case $h = (0,0,0,0,0)$, this implies that $C$ has 125 components, 
and so $d$ is an integer.  Thus Condition (ii) holds trivially.

If $h \neq (0,0,0,0,0)$, then $r=5$. 
First note that (i) implies that $r_i(h)+r_i(g) \, (\on{mod} 5)$ is the same
for any $i$. Thus, we only need to prove the statement for one $i$.
Let $\mu_5$ be the group generated by $[h]$.
Let $C' \cong \PP^1$ be one component of $C$ and 
let 
\[
 f' := \tilde{f} \big|_{C'} : C' \to \PP^4
\] 
be the $\mu_5$-equivariant morphism
induced from the $\bar{G}$-equivariant morphism $\tilde{f}: C \to \PP^4$.
$(f')^*(\sO(1))$ is a degree $5d$ line bundle on $C' =\PP^1$. 
Therefore, any lifting of the torus action on $\PP^1$ will have \emph{weights}
$(w, w+5d)$ at the fibers of the 2 fixed points.
Call these two fixed points $p'_1$ and $p'_2$.
Since $\mu_5 = \langle[h]\rangle$ is a subgroup of the torus,
the \emph{characters} of the $[h]$-action at the fibers of the 2 fixed points 
must be $(\zeta^w, \zeta^{w+5d})$, for some $w$ in $\{0, \ldots , 4\}$.

Let $q_1:= f'(p'_1)$ and $q_2 := f'(p'_2)$. 
By assumption, $q_1 \in \PP^4_h$, $q_2 \in \PP^4_g$.
Choose an $i \in I(h)$ and $j \in I(g)$ such that 
$i \neq j$, $x_i(q_1) \neq 0$ and $x_j(q_2) \neq 0$.
The action of $[h]$ on the fiber over $q_1$ and $q_2$ can be chosen to be
$(\zeta^{r_i(h)}, \zeta^{- r_j(h)})$.
By the above weight/character arguments,  
\[
 r_i(h) - (- r_j(h)) \equiv 5d \, (\on{mod} 5).
\]
Since $j \in I(g)$ and $i \in I(h)$, 
\[
 r_j(h) = r_j(h) - r_i(h) = r_i(g) - r_j(g) = r_i(g), 
\]
so we can rewrite the above as $r_i(h) + r_i(g) \equiv 5d \, (\on{mod} 5)$.

The nodal case follows similarly. Consider a nodal curve $f: \cC \to \cY$.  
Let $\cC_1, \ldots, \cC_n$ be the irreducible components connecting 
$p_1$ to $p_2$.  
It follows from Lemma~\ref{l:2.2}, each of these components will have 
2 orbifold points (at either nodes or marked points) and these will be the 
only points in $\cC$ with nontrivial orbifold structure.  
The above calculation for irreducible components plus the condition that 
all nodes be balanced in this situation then implies the claim.  
\end{proof}

Once condition (i) is satisfied, the degree of maps allowed is thus 
determined by the quantity \[ d(h,g) := \langle 
 (r_i(h) + r_i(g))/5 \rangle .\]  Note that this number remains constant as $i$ varies.

We will define generating functions related to the $J$-functions $J^\cY_g$ which 
isolate the $2$-point invariants of $ \sMbar_{0,h,g}(\cY, d)$.  Let 
\[
 S(d,h) :=\{(b,k)\,\,|\,\,\, 0 < b \leq d, \, \,\,\, 0 \leq k \leq 4, \,\,\,\,  
 \langle b \rangle = r_k(h)/5   \},
\] 
and let 
\[
  c(d, h) := \big|S(d,h)\big|.
\]

Given $h, g \in G$ such that $[h] = [g]^{-1}$, define 
\[ 
 Z_{h,g} := \sum_{d} Q^{c(d,h)} \sum_i \bigg\langle\frac{T^h_i}{z - \psi_1}, 
 \ii_g\bigg\rangle^\cY_{0,2,d} T^i_h,
\]
where $\{T_i^h\}$ is a basis for $H^*(\cY_h)$, and $\{T^i_h\}$ is the dual 
basis under the Chen-Ruan orbifold pairing.
(The motivation behind this choice of exponent for $Q$ will become clear 
in what follows: it is chosen to simplify the recursion satisfied by our 
generating function). 
Notice that by the above lemma, the only degrees which contribute to 
$Z_{h,g}$ are $d$ such that $\langle d \rangle = d(h,g)$.  
Finally, let 
\[
  Z_{g} := \ii_g + \sum \limits_{\{h | \,[h] = [g]^{-1}\}} Z_{h,g}.
\]  

Let $T = (\CC^*)^5$ (or $\CC^*$) act on $\CC^5$ with (generic) weights 
$-\lambda_0, \ldots , -\lambda_4$.  
This induces an action on $\PP^4$ and $\cY$.  
Furthermore there is an induced $T$-action on the inertia orbifold $I\cY$ 
and on $\sMbar_{0,2}(\cY, d)$.  
We will consider an equivariant analogue $Z^T_g$ of $Z_g$ defined by replacing
the coefficients of $Z_g$ with their 
equivariant counterparts:
\[
 Z^T_{h,g} := \sum_{d, i} Q^{c(d,h)} \bigg\langle\frac{T^h_i}{z - \psi_1}, 
 \ii_g\bigg\rangle^{\cY, T}_{0,2,d} T^i_h, \quad
Z^T_g := \ii_g + \sum \limits_{\{h | \,[h] = [g]^{-1}\}} Z^T_{h,g}.
\] 
where $\{T^h_i\}$ is now a basis of the equivariant cohomology $H^*_T(\cY_h)$. 

Consider the cohomology valued functions
\begin{equation} \label{e:2.2.2}
 Y^T_{h,g} :=  \sum_{\{d \,\vline \langle d \rangle = d(h,g)\}} Q^{c(d,h)} 
 \frac{\ii_{h^{-1}}}{\prod \limits_{(b,k)\in S(d,h)} (bz + H - \lambda_k)}, 
\end{equation}
where 
\[ 
  h^{-1} := (-r_0(h), \ldots, -r_4(h)) \,(\on{mod}5).
\]  
As with $Z$, let 
\begin{equation} \label{e:2.2.3}
 Y^T_g  := \ii_g + \sum \limits_{\{h | \,[h] = [g]^{-1}\}} Y^T_{h,g}.
\end{equation}

\begin{theorem}\label{t:Zformula}
We have the equality in equivariant cohomology: 
\[
 Z_g^T = Y^T_{g}.
\]
In particular, taking the nonequivariant limit, 
we conclude that $Z_g = Y_g$, 
(where $Y_g$ is the obvious non-equivariant limit of $Y^T_g$.)
\end{theorem}

\begin{remark}
For those who are familiar with the computation of the small
$J$-function for
toric manifolds \cite{aG2}, the generating functions $Z$, as indicated above, 
play the role of the $J$-function.
The hypergeometric-type functions $Y$ then take the place of the $I$-function.
Recall that one way of formulating the computation of genus zero GW invariants 
is to say that the $J$-function is equal to the $I$-function after 
a change of variables, called the \emph{mirror map}.
In the present case, the mirror map is trivial.
\end{remark}

\subsection{Proof of Theorem~\ref{t:Zformula}} \label{s:2.3}

The proof follows from a localization argument similar in spirit to that 
in \cite{aG2}.  
The strategy is to apply the Localization Theorem 
(after inverting the equivariant characters $\lambda_0, \ldots, \lambda_4$ in the 
ring $H^*_{CR,T}(\cY)$)
on the equivariant generating functions
to determine a recursion satisfied by $Z^T_g$.
This recursion relation in fact determines $Z^T_g$ up to the 
constant term in the Novikov variables.
We then show that $Y^T_g$ satisfies the same recursion.
Since $Z^T_g$ and $Y^T_g$ have the same initial term and the same recursion
relation, $Z^T_g = Y^T_g$.

\subsubsection{a lemma on $c(d,h)$}

We will first explain the seemingly strange appearance of the exponents 
$c(d,h)$ in the definition of $Z_{h,g}$.  

\begin{lemma}\label{l:dime}
Let 
\[
 m_d = \dim(\sMbar_{0,h,g}(\cY, d)),
\] 
then if $[h] = [g]^{-1}$ and $\langle d \rangle = d(h,g)$, we have 
\[
 c(d,h) = m_d - \dim(\cY_h) + 1.
\]
\end{lemma}

\begin{proof}  
The standard formula for virtual dimension gives 
\[
 m_d = 5d + 3 - \on{age}(h) - \on{age}(g).
\]  
Note that for any presentation 
$g = (r_0(g), \ldots, r_4(g))$, $\on{age}(g) = \sum_{i=0}^4 r_i(g)/5$.  
Because $[h] = [g]^{-1}$, we have that 
\[
  r_i(g) - r_j(g) \equiv r_j(h) - r_i(h) \, (\on{mod} 5).
\]
This allows us to write 
\[
  \frac{r_k(g)}{5} = \left\{ 
  \begin{array}{cc} - r_k(h)/5 + d(h,g) & d(h,g) \geq r_k(h)/5 \\ 
   1 - r_k(h)/5+ d(h,g) & d(h,g)< r_k(h)/5
  \end{array} \right.,
\] 
which gives 
\[
 \begin{array}{ll}
  m_d &= 5d + 3 - 5d(h,g) - |\{k\, | d(h,g) < r_k(h)/5\}| \\ 
  & = 5\lfloor d \rfloor + |\{k\, | d(h,g) \geq r_k(h)/5\}| - 2.
 \end{array}
\]
Now, for a fixed $k$, 
\[
  |\{b\, | 0 \leq b \leq d, \,\,\, \langle b \rangle =  r_k(h)/5  \}| 
  = \left\{ \begin{array}{cc} \lfloor d \rfloor & d(h,g) < r_k(h)/5 \\ 
    1 + \lfloor d \rfloor & d(h,g) \geq r_k(h)/5
  \end{array} \right\}.
\]  
Summing over all $k$, we get that 
\[
  m_d = |\{(b,k)\, | 0 \leq b \leq d, \quad 0 \leq k \leq 4,\quad 
  \langle b \rangle =  r_k(h)/5 \}| - 2.
\]  
Finally, 
\[
  \dim(\cY_g) = |\{k\, |\, 0 =  r_k(h)/5\}| - 1,
\] 
which gives the desired equality. 
\end{proof}

\subsubsection{Setting up the localization}
The action of $T$ on $\sMbar_{0,h,g}(\cY, d))$ allows us to reduce integrals 
on the moduli space to sums of integrals on the fixed point loci with respect 
to the torus action. 
As usual, this reduces us to considering integrals of certain graph sums 
(see \cite{tGrP}).  
The generating function $Z^T_g$ consists of integrals where the first 
insertion is the pull back of a class on 
\[
  \coprod \limits_{\{h| [h] = [g]^{-1}\}} \cY_h.
\]  
We will now express $Z_g$ in terms of a new basis for this space which 
interacts nicely with the localization procedure.  
For each coordinate $0 \leq i \leq 4$, $i$ is in $I(h)$ for exactly one $h$ 
in $\{h| [h] = [g]^{-1}\}$. 
(Recall that the presentations $h \in \{h| [h] = [g]^{-1}\}$ 
index the fixed point sets of $\PP^4$ with respect to $[h]$).  
Then for $i \in I(h)$, let $q_i$
be the $T$-fixed point of $\cY_h$ obtained by setting all coordinates 
$\{j \,\vline\, j \neq i\}$ equal to zero.  
Then, for $i \in I(h)$, let 
\[ 
  \phi_i =  \ii_h\cdot \prod \limits_{j \in I(h) - i} H - \lambda_j.
\]  
If we pair $Z^T_g$ with $\phi_i$, we obtain the function 
\begin{equation*} 
  Z^T_{i,g} = \frac{\delta^{i, I(g)}}{125} + \sum_d Q^{c(d, h)} \bigg\langle
  \frac{\phi_i}{z - \psi_1}, \ii_g\bigg\rangle^{\cY, T}_{0,2,d},
\end{equation*} 
where $\delta^{i, I(g)}$ equals $1$ if $i \in I(g)$ and $0$ otherwise. 
The fixed point set of $\cY_h$ consists of $\{q_j | j \in I(h)\}$.  
Note that under the inclusion $i_j: \{ q_j \} \to \cY_h$, 
$H$ pulls back to $\lambda_j$.  
Therefore $i_j^* (\phi_i) = 0$ unless $i = j$.  
From this we see that the coefficients of $Z^T_{i,g}$ consist of integrals 
over graphs such that the first marked point is mapped to $q_i$.

We divide the remaining graphs into \emph{two types}:  
those in which the first marked point is on a contracted component, 
and those in which the first marked point is on a noncontracted component.  

\begin{claim}
There is no contribution from graphs of the first type.
\end{claim}

\begin{proof}
The proof is a dimension count. We will show that the contributions
from graphs of the first type must contain as a multiplicative factor
integrals of the form $\int_{M} \Psi$ such that $\deg_\CC (\Psi) > \dim (M)$,
and hence the vanishing claim.

The complex degree of $\phi_i$ is $\dim(\cY_h)$, so the invariant 
$\langle\phi_i \psi_1^k, \ii_g\rangle^{\cY,T}_{0,2,d}$ vanishes unless 
$k \geq m_d - dim(\cY_h)$.  Thus we can simplify our expression for $Z^T_{i,g}$:
\begin{equation*} 
 \begin{split}
  Z^T_{i,g}&=\frac{\delta^{i, I(g)}}{125} + \sum_d Q^{c(d,h)}
   \bigg\langle\frac{\phi_i}{z - \psi_1}, \ii_g \bigg\rangle^{\cY,T}_{0,2,d} \\ 
  &= \frac{\delta^{i, I(g)}}{125} + \sum_d Q^{c(d,h)}\frac{1}{z} \sum_{k = 0}^\infty
   \big\langle \phi_i (\psi_1/z)^k, \ii_g\big\rangle^{\cY,T}_{0,2,d} \\ 
 &=\frac{\delta^{i, I(g)}}{125} +\sum_d Q^{c(d,h)}\frac{1}{z} \sum_{k = c(d,h)-1}^\infty
   \big\langle\phi_i (\psi_1/z)^k, \ii_g\big\rangle^{\cY,T}_{0,2,d} \\
  &=\frac{\delta^{i, I(g)}}{125} + \sum_d \Big(\frac{Q}{z} \Big)^{c(d,h)} 
   \bigg\langle
  \frac{\phi_i \psi_1^{c(d,h)-1}}{1 - (\psi_1/z)}, \ii_g\bigg\rangle^{\cY,T}_{0,2,d}.
 \end{split}
\end{equation*}
Here the third equality follows from Lemma \ref{l:dime}.  

Now consider a fixed point graph $M_\Gamma$ such that 
$p_1$ is on a contracted component.  
At the level of virtual classes, we can write 
\begin{equation} \label{e:2.3.1}
  \left[M_\Gamma\right] = F(\Gamma) \cdot \prod_k \left[M_{v_k}\right],
\end{equation}
where each $M_{v_k}$ represents a contracted component of the graph isomorphic 
to a component of $\Mbar_{0, n}(B\ZZ_r, 0)$, 
and $F(\Gamma)$ is a factor determined by $\Gamma$.  
Let $M_{v_0}$ be the component containing $p_1$.  $M_{v_0}$ contains at most 
$2$ orbifold marked points, and the number of non-orbifold marked points 
is restricted by $d$.  
In particular, each non-orbifold marked point corresponds to a 
(non-orbifold) edge of the dual graph.  
Each of these edges must have degree at least $1$, so if the total degree of 
the map is $d$, then there can be at most $\lfloor d \rfloor $ nontwisted 
marked points.  
Thus the dimension of $M_{v_0}$ is at most $\lfloor d\rfloor - 1$.  
Now, the proof of Lemma \ref{l:dime} shows that 
\[
 c(d, h) - 1 = 5\lfloor d \rfloor + |\{k\,|\, r_k(h)/5 \leq d(h,g)\}| 
 - 2 - \dim(\cY_h).
\]  
But $\dim(\cY_h)$ is exactly $|\{k\,|\, r_k(h) =0 \}| - 1$, 
which implies that $$c(d, h) - 1 \geq 5 \lfloor d \rfloor - 1.$$ 
If $d \geq 1$, the above quantity is strictly greater than 
$\lfloor d \rfloor - 1$.  
Because there do not exist graphs such that $p_1$ is on a non-contracted 
component for $d<1$, we have that for $M_\Gamma$, 
$c(d, h) - 1 \gneq \dim(M_{v_0})$.  
But $\psi_1^{c(d, I) - 1}$ must therefore vanish on these graphs, 
proving the claim.
\end{proof}

\subsubsection{Contributions from a graph of the second type}\label{s:2.3.3}

Now let us consider the contribution to 
$\langle\frac{\phi_i}{z - \psi_1}, \ii_g\rangle^{\cY,T}_{0,2,d}$ 
from a particular graph $\Gamma$ of the second type.
In particular, we know that $p_1$ is on a noncontracted component.  
Call this component $\cC_0$, and denote the rest of the graph $\Gamma '$.  
$\Gamma'$ and $\cC_0$ connect at a node $p'$, which maps to some $q_k \in \cY$.
Let $d'$ be the degree of one connected component of the principal 
$\bar{G}$-bundle above $\cC_0$.  
We know from Proposition~\ref{p:mapsallowed} that 
$\langle d' \rangle = r_k(h)/5$.  
By identifying $p' \in \Gamma '$ as a marked point 
(replacing $p_1$ on $\cC_0$), we can view $M_{\Gamma '}$ as a fixed point 
locus in $\sMbar_{0, h', g}(\cY, d - d')$, where 
$[h] = [h']$, but $r_k(h') = 0$.  
Our plan will be to express integrals on $M_\Gamma$ in terms of integrals on 
$M_{\Gamma '}$, thus reducing the calculation to one involving maps of 
strictly smaller degree.
This will give us a recursion.

The factor $F(\Gamma)$ in Equation~\ref{e:2.3.1}
is composed of three contributions:  the automorphisms of the graph $\Gamma$ 
itself, a contribution from each edge of $\Gamma$ 
(the non-contracted components of curves in $M_\Gamma$), 
and a contribution from certain flags of $\Gamma$ 
(the nodes of curves in $M_\Gamma$).  
The edge corresponding to $\cC_0$ maps to the line 
$q_{ik} \cong \PP^1/\bar{G}$ connecting $q_i$ and $q_k$.  
(Note that the $\bar{G}$-action is a subgroup of the big torus $(\CC^*)^4$ of 
$\PP^4$, $\bar{G}$ naturally acts on $(\CC^*)^4$ orbits.)
The degree of the map upstairs is $5 d'$.  
Thus there is a contribution of $1/(5 d')$ to $F(\Gamma)$ from the 
automorphism of $M_\Gamma$ coming from rotating the underlying curve.  
The edge also contributes a factor of $1/25$ due to the fact that $q_{ik}$ is 
a $(\ZZ/5\ZZ)^2$-gerbe.  
So the total contribution to $F(\Gamma)$ from the edge containing $p_1$ is 
$1/(125d')$.  
The contribution from the node $p'$ is $125/r$.  
(Recall $r = |[h]|$, which is equal to the order of the isotropy at $p'$).  
There will be an additional factor of $r$ appearing when we examine 
deformations of $M_\Gamma$, thus canceling the $r$ in the denominator.  
We finally arrive at the relation 
\[
  \left[M_\Gamma\right] = F(\Gamma) \cdot \prod_{\text{vertices }v \in \Gamma} 
  \left[M_{v}\right] \, = \, \frac{F(\Gamma ')}{d'} \cdot
  \prod_{\text{vertices }v \in \Gamma '} \left[M_{v}\right] 
  \, = \, \frac{1}{d'}\left[M_{\Gamma '}\right].
\]
By examining the localization exact sequence (see \cite{tGrP}), 
we have the following 
identity: 
\begin{equation}\label{e:norm}
  e(N_\Gamma) 
  = \frac{e(H^0(\cC_0, f^*T\cY)^m)(\text{node smoothing at $p'$})}
   {e(H^0(p', f^*T\cY)^m)e(H^1(\cC_0, f^*T\cY)^m) e((H^0(\cC_0, T\cC_0)^m)
   }e(N_{\Gamma '})
\end{equation} 
where $e$ denotes the equivariant Euler class, and as is standard 
we identify certain vector bundles with their fibers.  
Here the superscript $m$ denotes the moving part of the vector bundle with 
respect to the torus action.
Let us calculate the factors in \eqref{e:norm}.

$\bullet \, (\text{node smoothing at $p'$})$:
The node smoothing contributes a factor of 
\[
 \left( \frac{\lambda_k - \lambda_i}{rd'} - \frac{\psi_1 '}{r} \right) = 
 \frac{1}{r}\left( \frac{\lambda_k - \lambda_i}{d'} - \psi_1 ' \right),
\] 
where $\psi_1 '$ is the $\psi$-class corresponding to $p_1 '$ on $M_{\Gamma} '$. 
This factor of $r$ is what cancels with the previous factor mentioned above.

$\bullet \, e(H^0(\cC_0, T\cC_0)^m)$: 
Let $C$ be the principal $\bar{G}$-bundle over $\cC_0$ induced from
$f|_{\cC_0}: \cC_0 \to [\PP^4/\bar{G}]$. 
As was argued in Proposition~\ref{p:mapsallowed}, 
$C$ consists of $(|\bar{G}|/r)$ copies of $\PP^1$.  
Let $C_0$ be one of these copies.  
Then $C_0$ is a principal $\langle [h]\rangle$-bundle over $\cC_0$ and
\[
 H^0(\cC_0, T\cC_0) = H^0(C_0, TC_0 )^{\langle [h]\rangle}.
\]
The $\langle [h] \rangle$-invariant part of $H^0(C_0, TC_0 )$ is one 
dimensional.  
It is fixed by the torus action, thus the moving part of $H^0(\cC_0, T\cC_0)$ 
is trivial
and $e(H^0(\cC_0, T\cC_0)^m) = 1$.

$\bullet \, e(H^1(\cC_0, f^*T\cY)^m)$: Let $C_0$ be as in the previous bullet, 
then
\[
 H^1(\cC_0, f^*T\cY) = H^1(C_0, \tilde{f}^*T\PP^4 )^{\langle [h]\rangle} = 0.
\]
Therefore $e(H^1(\cC_0, f^*T\cY)^m)=1$.  

$\bullet \, e(H^0(\cC_0, f^*T\cY)^m)$:
To calculate this term, note that
\[
 H^0(\cC_0, f^*T\cY)^m \cong
 \left( H^0(C_0, \tilde{f}^*T\PP^4)^{\langle [h] \rangle}\right)^m.
\]
We will look at the $\langle [h] \rangle$ invariant 
part of the short exact sequence 
\[
 0 \to \CC \to H^0(\sO_{C_0}(r d' )) \otimes V \to H^0(\tilde{f}^*T\PP^4) 
  \to 0,
\]
where $\PP^4 = \PP(V)$ and $V \cong \CC^5$.
The exact sequence comes from the pullback of the Euler sequence for 
$\PP^4$ to $C_0$.
(Note that the degree of $\tilde{f}: C_0 \to \PP^4$ is $rd'$).   
The action of $[h]$ on the first term in the sequence is trivial.  

Recall that $\PP(V)$ has coordinates $[x_0, \ldots, x_4]$.  
Let $[s,t]$ be homogeneous coordinates on $C_0 \cong \PP^1$, 
such that the preimage of $p_1$ in $C_0$ is $[0,1]$ and 
the preimage of $p'$ in $C_0$ is $[1,0]$. 
Then the middle term of the sequence is spanned by elements of the form 
$s^at^b \frac{\partial}{\partial x_l}$ where 
$0 \leq l \leq 4$ and $a + b = rd'$.  
The action is given by 
\[
 [h] . (s^at^b \frac{\partial}{\partial x_l}) 
 = e^{2 \pi \sqrt{-1} (-a + r_l(h))/r}s^at^b \frac{\partial}{\partial x_l},
\]
and so this summand is invariant under the $\langle [h] \rangle$-action 
if and only if $r_l(h)/r = \langle a/r \rangle$.  
The $\CC^*$-action on this term has weight 
\[
\left(a/r d'\right)\lambda_k + \left(b/r d'\right)\lambda_i - \lambda_l,
\] 
so we finally arrive at 
\[
 \begin{split} &e(H^0(\cC_0, f^*T\cY)^m) \\
 = & \prod_{\substack{\{(a,l)| 0 \leq a \leq rd'\,\, 0 \leq l \leq 4\,\, r_l(h)/r = \langle a/r \rangle \}\\ \setminus \{(0,i),\,\,(r d',k)\}}}
 \left(\frac{a}{rd'}\lambda_k + \frac{rd'-a}{rd'}\lambda_i - \lambda_l\right) \\ 
 =&\prod_{\substack{\{(a,l)| 0 \leq a \leq rd'\,\, 0 \leq l \leq 4\,\, r_l(h)/r = \langle a/r \rangle \}\\ \setminus \{(0,i),\,\,(r d',k)\}}}
 \left(a\left(\frac{\lambda_k  -\lambda_i}{rd'}\right) 
 + \lambda_i - \lambda_l\right).
 \end{split}
\]

$\bullet \, e(H^0(p', f^*T\cY)^m )$:
Similarly, the node $p'$ is isomorphic to $B\ZZ_r$, and each of the 
$|\bar{G}|/r$ points lying in the principal $\bar{G}$-bundle over $p'$ 
is a principal $\langle [h] \rangle$-bundle over $p'$.  
Thus $H^0(p', f^*T\cY)^m \cong \left( (T_{q_k}\PP^n)^{\langle [h] \rangle}\right)^m$ 
and 
\[
 e(H^0(p', f^*T\cY)^m ) = \prod_{l \in I(h') \setminus\{ k \}}
  \left(\lambda_k - \lambda_l\right).
\]

Finally note that $ev_1^*(\phi_i) = \prod_{l \in I(h) - i} (\lambda_i - \lambda_l)$.
We can do one further simplification.  
On the graphs which we consider, namely those where $p_1$ is on a noncontracted 
component, $\psi_1$ restricts to $\frac{\lambda_k - \lambda_i}{d'}$.  
(In fact $e(T^*_{p_1} \cC) \cong \frac{\lambda_k - \lambda_i}{rd'}$, but because
we are following the convention that $\psi$-classes are pulled back from the 
reification, we must multiply this by a factor of $r$).

These calculations plus \eqref{e:norm} then give us the contribution to 
$\big\langle\frac{\phi_i \psi_1^{c(d,h) - 1}}{1 - \psi_1/z}, \ii_{g} 
\big\rangle^{\cY,T}_{0,2,d}$ from the graph $M_\Gamma$: 
\[
\begin{split}
 &\int_{[M_\Gamma]} \frac{ev_1^*(\phi_i)\psi_1^{c(d,I) - 1}}
  {e(N_\Gamma)\left(1 - \psi_1/z \right)} \\
 = &\frac{\frac{\lambda_k - \lambda_i}{d'}^{c(d,I) - 1} 
    \prod_{l \in I(h) \setminus \{ i\}} (\lambda_i - \lambda_l)e(H^1(\cC_0, f^*T\cY)^m)}
 {e(H^0(\cC_0, f^*T\cY)^m) (1 - \frac{\lambda_k - \lambda_i}{d'z})} \\ 
 &\cdot\frac{1}{d '}\int_{[M_\Gamma ']} \frac{e(H^0(p', f^*T\cY)^m)}
   {(\text{node smoothing at $p'$})e(N_{\Gamma'})} \\ 
 =&\frac{\frac{\lambda_k - \lambda_i}{d'}^{c(d,h) - 1}\prod_{l \in I(h)\setminus \{ i\}} 
   (\lambda_i - \lambda_l)}{(d' - \frac{\lambda_k - \lambda_i}{z}) 
  \prod \limits_{\substack{\{(a,l)| 0 \leq a \leq rd'\,\, 0 \leq l \leq 4\,\, r_l(h)/r = \langle a/r \rangle \}\\ \setminus \{(0,i),\,\,(r d',k)\}}}
  \left(a\left(\frac{\lambda_k  -\lambda_i}{rd'}\right) + \lambda_i - \lambda_l\right)} \\  
 &\cdot \int_{[M_{\Gamma'}]} \frac{\prod_{l \in I(h')\setminus \{ k \}}
   \left(\lambda_k - \lambda_l\right)}{(\frac{\lambda_k - \lambda_i}{d'} - \psi_1)
   e(N_{\Gamma'})}.
\end{split}
\]

\subsubsection{Recursion relations} \label{s:2.3.4}
We will formulate the above computations into a recursion relation.
To do that, the following regularity lemma is needed.

\begin{lemma}[Regularity Lemma]
$Z^T_{i,g}$ is an element of $\QQ(\lambda_i,z)[[Q]]$.
The coefficient of each $Q^D$ is a rational function of $\lambda_i$ and $z$
which is regular at $z = (\lambda_i -\lambda_j)/k$ for all $j\neq i$ and 
$k \geq 1$.
\end{lemma}

\begin{proof}  This follows from 
a standard localization argument, see e.g.\ Lemma 11.2.8
in \cite{CK}.
\end{proof}

Using the Regularity Lemma, the above computation simplifies to 
\[\begin{split}
&\left( \bigg\langle\frac{\phi_i \psi_1^{c(d, h) - 1}}{1 - \psi_1/z}, \ii_{g}\bigg\rangle^{\cY,T}_{0,2,d} \right)_{M_{\Gamma}} \\
&= C^{i,k}_{d'} \cdot \left(\frac{\lambda_k - \lambda_i}{d'}\right)^{c(d,h) - 1 - (c(d', h) - 1)} \cdot \left( \bigg\langle\frac{\phi_k}{z - \psi_1}, \ii_{g}\bigg\rangle^{\cY,T}_{0,2,d - d'} \right)_{M_{\Gamma'}} \bigg|_{z \mapsto \frac{\lambda_k - \lambda_i}{d'}},
\end{split}
\] 
where 
\[
 C^{i,k}_{d'} = \frac{1}{(d' - \frac{\lambda_k - \lambda_i}{z})
   \prod \limits_{\{(a,l)\in S(d', h) \setminus \{ (d',k) \}\}}
\left(a + d'\left(\frac{\lambda_i - \lambda_l}{\lambda_k  -\lambda_i}\right)\right)}
\] 
and 
$\left(-
\right)_{M_{\Gamma}}$
means the contribution of the fixed component $M_{{\Gamma}}$ 
to the expression in parentheses.

Due to the fact that $r_k(h)/5 = \langle d' \rangle$, one can check that 
\[
 c(d, h) - c(d', h) = c(d - d', h')
\] 
(see \eqref{e:recur}). 
We arrive at the expression 
\[
  C^{i,k}_{d'} \cdot \left( Q^{c(d-d', k)}   \bigg\langle\frac{\phi_k}{z - \psi_1},
  \ii_{g}\bigg\rangle^{\cY,T}_{0,2,d - d'} \right)_{M_{\Gamma'}} 
  \bigg|_{z \mapsto \frac{\lambda_k - \lambda_i}{d'}, Q \mapsto \frac{\lambda_k - \lambda_i}{d'} }.
\]
After summing over all possible graphs, we obtain the recursion: 
\begin{equation}\label{e:recursion}
 Z^T_{i,g} = \frac{\delta^{i, I(g)}}{125} + 
  \sum_{\{(d',k) | \frac{r_k(h)}{5} = \langle d'\rangle , k \neq i, d' \neq 0 \}} 
  \left(\frac{Q}{z}\right)^{c(d', h)}C^{i,k}_{d'}\cdot Z^T_{k, g}
   \bigg|_{z \mapsto \frac{\lambda_k - \lambda_i}{d'}, Q \mapsto \frac{Q}{z}\frac{\lambda_k - \lambda_i}{d'} }.
\end{equation}
Although we have suppressed this in the notation, 
recall that in the above summand, $h$ is the presentation such that 
$\phi_i \in \cY_h$ ($i \in I(h)$).

\vspace{10pt}

We will now turn our attention to $Y^T_g$.
Let us define the function $Y^T_{i,g}$ analogously to that of 
$Z^T_{i,g}$ , 
\[
  Y^T_{i,g} := ( \phi_i, Y^T_g )^\cY_{CR}.
\]  
For $i \in I(h)$, 
\[
 Y^T_{i,g} = \frac{1}{125}\left(\delta^{i, I(g)} 
 + \sum_{\langle d \rangle = d(h,g)} Q^{c(d,h)} 
 \frac{1}{\prod \limits_{(b,k) \in S(d, h)} (bz + \lambda_i - \lambda_k)}\right).
\]
\begin{claim}
$Y^T_{i,g}$ satisfy the same recursion as $Z^T_{i,g}$ in \eqref{e:recursion}.
\end{claim} 
\begin{proof}
Consider the summand of $Y^T_{i,g}$ of degree $c(d,h)$ in $Q$, which we will denote 
$(Y^T_{i,g})^{c(d,h)}$.
\[
\begin{split}
  (Y^T_{i,g})^{c(d,h)}
 = &\frac{1}{125} \left(\frac{Q}{z} \right)^{c(d,h)}
 \frac{1}{\prod_{(b,k) \in S(d,h)} 
   \left(b + (\lambda_i - \lambda_k)/z\right)} \\
 =&\frac{1}{125} \left(\frac{Q}{z} \right)^{c(d,h)}
  \sum_{\{(b,k) | r_k(h)/5 = \langle b \rangle , k \neq i, b\neq 0 \}}
 \frac{1}{\left(b + (\lambda_i - \lambda_k)/z\right)} \\
 &\qquad \cdot \frac{1}{\prod \limits_{(m,l) \in S(d,h) \setminus \{ (b,k) \} }
   \left( b (\lambda_i - \lambda_l)/(\lambda_k - \lambda_i) + m\right)}\\ 
 =&\frac{1}{125} \left(\frac{Q}{z} \right)^{c(d,h)} 
  \sum_{\{(b,k) | r_k(h)/5 = \langle b \rangle , k \neq i, b \neq 0 \}}  \\
 &\left( \frac{ 1/\left(b + (\lambda_i - \lambda_k)/z\right)}
       {\prod \limits_{\{(m,l) \in S(d,h) \setminus \{(b,k)\}| m \leq b\}}
   \left( b (\lambda_i - \lambda_l)/(\lambda_k - \lambda_i) + m\right)} \right. \\ 
 &\qquad \cdot  \left.
       \frac{1}{\prod \limits_{\{(m,l) \in S(d,h) \setminus \{(b,k)\}| m > b\}}
  \left( b (\lambda_i - \lambda_l)/(\lambda_k - \lambda_i) + m\right)} \right).
\end{split}
\]

The last product from above can be rewritten as 
\[ 
  \prod_{(n,l) \in S(d - b, h')} \left(n +  
  b\frac{\lambda_k - \lambda_l}{\lambda_k - \lambda_i}\right), 
\] 
where $h'$ is chosen such that $[h] = [h']$ and $k \in I(h')$.  
To see this note that if $(b, k)$ and $(m,l)$ are both in $S(d, h)$, 
then by definition $r_k(h)/5 = \langle b \rangle$ and 
$r_l(h)/5 = \langle m \rangle$.  If $k \in I(h')$, then 
\[
 \begin{split}
  &\frac{r_l(h')}{5} = \frac{r_l(h')}{5} - \frac{r_k(h')}{5} \\
  \equiv &\frac{r_l(h)}{5} - \frac{r_k(h)}{5} \equiv \langle m \rangle - 
  \langle b \rangle \equiv \langle m - b \rangle \, (\on{mod} 1).
 \end{split}
\]  
In other words $r_l(h')/5 = \langle m - b \rangle$.  This proves that if $(b, k)\in S(d, h)$, and $h'$ is chosen as above, then for pairs $(m, l)$ with $b < m \leq d$, 
\begin{equation}\label{e:recur}
(m,l) \in S(d, h)\text{ if and only if }(m - b, l) \in S(d - b, h').
\end{equation}
We arrive at the relation
\begin{equation*}
\begin{split}
 &\left(Y^T_{i,g}\right)^{c(d,h)} \\
 \quad =  &\sum_{\{(b,k) | r_k(h)/5 = \langle b\rangle , k \neq i, b \neq 0 \}} 
 \left(\frac{Q}{z}\right)^{c(b, h)}C^{i,k}_{b}\left(Y^T_{k, g}\right)^{c(d-b, h')}
\bigg|_{z \mapsto \frac{\lambda_k - \lambda_i}{b}, Q \mapsto \frac{Q}{z}\frac{\lambda_k - \lambda_i}{b}}.
\end{split}
\end{equation*}
We conclude that $Y^T_{i,g}$ satisfy the same recursion as $Z^T_{i,g}$.
\end{proof}

The recursion relation and initial conditions imply $Y^T_{i,g} = Z^T_{i,g}$.
The proof of Theorem~\ref{t:Zformula} is now complete.
\begin{remark}
As a corollary one may easily obtain an 
explicit formula for the small $J$-matrix $J^\cY_{small}(t, z)$ by isolating
coefficients of the various $Z^\cY_g$.  We give an explicit expression
for certain specified rows of $J^\cY_{small}(t, z)$ in 
Corollary~\ref{c:ambientJformula}.
\end{remark}

\section{$A$ model of the mirror quintic $\cW$} \label{s:3}

\subsection{Fermat quintic and its mirror} \label{s:3.1}

Let $M \subset \PP^4$ be the Fermat quintic defined by the equation
$Q_0(x) = x_0^5 + x_1^5 + x_3^5 + x_4^5 + x_5^5$
\[
  M := \{ Q_0(x) = 0 \} \subset \PP^4.
\]
The Greene--Plesser's \emph{mirror construction} \cite{bGrP} gives the 
\emph{mirror orbifold} as the quotient stack
\[
  \cW := [M / \bar{G}].
\]
Note that the $\bar{G}$-action on $\PP^4$ \eqref{e:G}
preserves the quintic equation $Q_0(x)$ and therefore induces an action on $M$.
Equivalently, 
\begin{equation} \label{e:3.1.1}
 \cW = \{ Q_0 =0 \} \subset \cY = [ \PP^4/\bar{G}].
\end{equation}

\begin{remark} \label{r:3.1}
Since in this section we will only be interested in the Gromov--Witten theory 
($A$ model), which is deformation invariant,
we will only speak of the mirror orbifold instead of the mirror family.
\end{remark}

Recall in Lemma~\ref{l:2.1} the inertia orbifold of $\cY = [\PP^4/\bar{G}]$ 
is indexed by $g \in G$.
For a particular $g$, the dimension of $\cY_g$ is equal to 
$\big|\{j | r_j = 0\}\big| - 1$, and can be identified with a linear subspace of $\cY$.
The age shift of $\cY_g$ is $\on{age}(g) = \sum_{i = 0}^4 r_i/5$.  

The inertia orbifold of the mirror quintic $\cW$ can be described by 
that of $\cY$.
$\cW$ intersects nontrivially with $\cY_g$ exactly when 
$\big|\{j | r_j = 0 \}\big| \geq 2$. (that is, $\dim \cY_g \geq 1$.) 
Let 
\[ 
 \bar{S} := \left\{g= (r_0, \ldots, r_4) \in G \big|\, 2 \leq \big|\{j | r_j = 0 \}\big| \, \right\}.  
\]
(Note that $\bar{S}$ contains $e=(0, \ldots, 0)$.) Then 
\[
  I\cW = \coprod_{g \in \bar{S}} \cW_g \,, \qquad 
   \cW_g := \cW \cap \cY_g.  
\] 
All nontrivial intersections are transverse, so 
\[
 \dim(\cW_g) = \dim(\cY_g) - 1 = \big|\{j | r_j = 0\}\big| - 2.
\]
It follows that the age shift of $\cW_g$ is equal to the age shift of $\cY_g$.  
The cohomology of $\cW$ is given by 
\[ 
 H^*_{CR}(\cW) = \bigoplus_{g \in \bar{S}} H^{* - 2\on{age}(g)}(\cW_g).
\]

In the sequel, we will only be interested in the subring of $H^*_{CR}(\cW)$
consisting of classes of even (real) degree.  We will denote this 
ring as $H^{even}_{CR}(\cW)$.  It can be checked via a direct
calculation that if
 $i: \cW \hookrightarrow \cY$ is the inclusion, 
\[
 H_{CR}^{even}(\cW) = i^* H^*_{CR} (\cY).
\]

\begin{conventions} \label{conv:2}
Let $H$ be the hyperplane class on $\PP^4$.
By an abuse of notation, we will denote $H$ any fixed choice of $L$ on $\cY$ 
such that $\pi_{\PP^4}^*(L) = H$, 
where $\pi_{\PP^4}$ was defined in \eqref{e:cover}.
We will also denote $H$ the induced class on $\cW$.
Even though there are as many as $|\bar{G}|$ choices of $L$,
they are topologically equivalent and will serve the same purpose in
our discussion.
\end{conventions}

A convenient basis $\{T_i\}$ for $H_{CR}^{even}(\cW)$ is 
\begin{equation} \label{e:Wbasis}
  \bigcup_{g \in \bar{S}}   \{\ii_g, \ii_g {H}, \ldots , 
  \ii_g {H}^{\dim(\cW_g)}\}.
\end{equation}

We also note that $H_{CR}^{even}(\cW) \subset H^*_{CR}(\cW)$ is a self-dual subring
with respect to the Poincar\'e pairing of $H^*_{CR}(\cW)$.
Furthermore, this basis is self-dual (up to a constant factor).
Given $g = (r_0, \ldots, r_4) \in S$, let 
\[
 g^{-1}:= (-r_1, \ldots, -r_4) \, (\on{mod} 5).  
\]
Then the Poincar\'e dual elements can be easily calculated:
\[
 \left( \ii_g {H}^k \right)^{\vee} 
 = 25 \left( \ii_{g^{-1}} {H}^{\dim(\cW_g) - k} \right).
\]

\subsection{$J$-functions of $\cW$} \label{s:3.2}

\begin{conventions} \label{conv:3}
By the matrix $J$-function of $\cW$, we will mean the matrix consisting of
the collection of $H^{even}_{CR}(\cW)$-valued functions
with variable $\bt = t H$.
\begin{equation}\label{e:sJW} 
 J_g^{\cW}(t,z) := e^{t{H}/z}\left(\ii_g + \sum_{d , i} q^d e^{dt} 
 \bigg\langle \frac{T_i}{z - \psi_1}, \ii_g \bigg\rangle^{\cW}_{0,2,d} 
 T^i\right),
\end{equation}
where the basis $\{T_i\}$ is for $H^{even}_{CR}(\cW)$, as in \eqref{e:Wbasis}.
Here as in Section~\ref{s:2}, by the degree $d$ of a map $f: \cC \to \cW$ we mean 
\[ 
  d:= \int_{\cC} f^*({H}).
\] 
Note that if we extend the basis $\{T_i\}$ to full basis of $H^*_{CR}(\cC)$,
the classes of odd (real) degree will not contribute to 
$J^{\cW}_g(t, z)$, and thus \eqref{e:sJW} is equal to the $J_g$-function of
\eqref{e:sJ}.
\end{conventions}

As has been shown in Proposition~\ref{p:mapsallowed}, for an orbi-curve 
$\cC$ with two marked points, the degree must be a multiple of $1/5$.
Recall also from Proposition~\ref{p:mapsallowed} that the only nonzero 
contribution to the terms in $J_g^{\cW}$ comes from elements $T_i$ supported 
on some $\cW_h$ such that $[h] = [g^{-1}]$.  
From the definition of $\bar{S}$, it is required that
\begin{equation} \label{e:3.1.2}
  \big|\{j | r_j = 0 \}\big| \geq 2, \quad 
  \sum r_j \equiv 0 \, (\on{mod} 5).
\end{equation}
We will enumerate all possible cases.

It follows from the conditions \eqref{e:3.1.2} that
$\big|\{j | r_j = 0 \}\big|$ must be equal to $2$, $3$ or $5$.
That is, $\dim(\cW_g)$ is equal to $0$, $1$ or $3$.

If $\dim(\cW_g) = 3$, $g = e = (0,0,0,0,0)$ and $\ii_e = 1$. 
The only basis elements which contribute to $J_e^{\cW}$ come from the 
nontwisted sector.  We have 
\begin{equation} \label{e:dim3}
 J_e^{\cW}(t,z) = e^{t{H}/z}\left(1 + \sum_{d > 0} q^d e^{dt} \bigg\langle 
 \frac{{H}^i}{z - \psi_1}, 1 \bigg\rangle^{\cW}_{0,2,d} 
 (25{H}^{3-i})\right). 
\end{equation} 

If $\dim(\cW_g) = 1$, then up to a permutation of the entries, 
$g = (0,0,0, r_1, r_2)$ with $r_1 \neq r_2$. 
By definition of $\bar{S}$, other than $g$ there is no 
$h \in \bar{S}$ such that $[h] = [g]$. 
Therefore, the two basis elements which contribute nontrivially to 
$J_g^{\cW}$ are $\ii_{g^{-1}}$ and $\ii_{g^{-1}} {H}$.  We arrive at 
\begin{equation} \label{e:dim1}
 \begin{split}
  &J_g^{\cW}(t,z) = e^{t {H}/z} \Bigg( \ii_g +  \\
  &\sum_{d > 0} q^d e^{dt} \left(\bigg\langle \frac{\ii_{g^{-1}}}{z - \psi_1}, \ii_g 
  \bigg\rangle^{\cW}_{0,2,d} (25\ii_g{H}) + \bigg\langle 
  \frac{\ii_{g^{-1}}{H}}{z - \psi_1}, \ii_g \bigg\rangle^{\cW}_{0,2,d} 
  (25\ii_{g})\right) \Bigg).
 \end{split}
\end{equation}  

If $\dim(\cW_g) = 0$, then up to a permutation of the entries, 
$g = (0,0,r_1, r_1, r_2)$, with $r_1 \neq r_2$.
There is only one 
other $g_1 \in \bar{S}$ such that $[g_1] =[g]$, 
namely, $g_1 = (-r_1, -r_1, 0,0, r_2 - r_1) \, (\on{mod} 5)$.  
The two basis elements which contribute nontrivially to the invariants of 
$J_g^{\cW}$ are $\ii_{g^{-1}}$ and $\ii_{(g_1)^{-1}}$.  
Thus we can express $J_g^{\cW}(t,z)$ as 
\begin{equation} \label{e:dim0} 
 \begin{split}
  &J_g^{\cW}(t,z) = e^{t {H}/z} \Bigg( \ii_g +  \\
  &\sum_{d > 0} q^d e^{dt} 
  \left(\bigg\langle \frac{\ii_{g^{-1}}}{z - \psi_1}, \ii_g 
  \bigg\rangle^{\cW}_{0,2,d} (25\ii_g) + \bigg\langle 
  \frac{\ii_{(g_1)^{-1}}}{z - \psi_1}, \ii_g \bigg\rangle^{\cW}_{0,2,d} 
  (25\ii_{g_1})\right)\Bigg).
 \end{split}
\end{equation}
Thus for each twisted component $\cW_g$, the $J$-function $J_g^{\cW}$ 
has two components.  

\vspace{10pt}

We will relate the functions $J_g^{\cW}$ to certain hypergeometric functions, 
called $I$-functions.
To start with, let us introduce ``bundled-twisted'' Gromov--Witten 
invariants. Let $E \to \cX$ be a line bundle over the orbifold $\cX$.
We have the following diagram
\[
 \begin{CD}
      @.    E\\
     @.     @VVV \\
  \cC @>f>> \cX  \\
  @VV{\pi}V @. \\
  \sMbar_{0,n}(\cX, d). 
 \end{CD}
\]

The $E$-twisted Gromov--Witten invariants are defined to be
\[
 \big\langle \alpha_1 \psi^{k_1}, \ldots , \alpha_n \psi^{k_n}
 \big\rangle_{0,n,d}^{\cX, \on{tw}}  =  \int_{[\sMbar_{0,n}(\cX, d)]^{vir}} 
 \prod_{i=1}^n ev_i^*(\alpha_i) \psi_i^{k_i} \cup e(E_{0,n,d}),
\]
where
\[
 E_{o,n,d} := \pi_* f^*(E)
\]
and $e(E_{0,n,d})$ is the Euler class of the $K$-class.
We can define a twisted pairing on $H^*_{CR}(\cX; \Lambda)$ by 
\[
  ( \alpha_1, \alpha_2 )_{CR}^{\cX, \on{tw}} = 
   \int_\cX \alpha_1 \cup I^*(\alpha_2) \cup e(E).
\]  
With this, we can define a twisted $J$-function 
\[
 J^{\cX, \on{tw}}(\bt,z) =  1 + \bt/z + \sum_{d} 
  \sum_{n \geq 0} \sum_i \frac{q^d}{n!}  \bigg\langle  \frac{T_i}{z - \psi_1}, 1, \bt, 
   \ldots, \bt \bigg\rangle_{0,2+k, d}^{\cX, \on{tw}} T^i. 
\]  
Here $T_i$ is a basis for $H^*_{CR}(\cX; \Lambda)$ and $T^i$ is the dual basis 
with respect to the twisted pairing.  

The twisted invariants are related to invariants on the hypersurface.
In our case, $\cX = \cY = [\PP^4/\bar{G}]$, and
$E = \sO (5) \to \cY$.
It is easy to see 
that $E_{0,n,d} = R^0 \pi_* f^*(\sO (5))$ is a vector bundle.
The embedding $i: \cW \hookrightarrow \cY$ induces a morphism
$\iota: \sMbar_{0,n}(\cW, d)\hookrightarrow \sMbar_{0,n}(\cY, d)$.
It is well-known that
\begin{equation} \label{3.2.6}
 \iota_* [\sMbar_{0,n}(\cW, d)]^{\on{vir}} = 
 e(E_{0,n,d}) \cap [\sMbar_{0,n}(\cY, d)]^{\on{vir}}.
\end{equation}
A proof can be found in e.g.~\cite{CKL}. 
(That proof, given in the nonorbifold setting there, 
can be readily modified to the orbifold setting.)
This relates the twisted invariants on $\cY$ to the invariants on $\cW$.  
Assume that $\bt$ is restricted to $H^{even}_{CR}(\cY)$, then
\[
J^\cW(\bt, z) = i^* J^{\cY, tw}(\bt, z).
\]
Let us now further restrict $\bt$ to $H^{2}_{CR}(\cY)$.  
In our setting we may write an element of $H^{2}_{CR}(\cY)$ as 
\begin{equation} \label{e:3.2.7}
 \bt = tH + \sum\limits_{\{g| \on{age}(g) = 1\}} t^g \ii_g.
\end{equation}
Write the $J$-function of $\cY$ as
\[
 J^{\cY} (\bt) = \sum_d q^d J_d^{\cY}(\bt).
\]
For each $d$, define the modification factor
\[
 M^{E/\cW}_d := \prod_{m=1}^{5d} (5H + m z).
\]
(Note that we have taken the $\lambda =0$ limit in \cite{CCIT}.)

\begin{definition} \label{d:Ifunction}
Define the \emph{twisted $I$-function} by
\[
 I^{E} (\bt) := \sum_d q^d M^{E/\cW}_d J_d^{\cY}(\bt)
\]
Write
\begin{equation} \label{e:3.2.8}
 \begin{split}
  I^{E} (\bt, z) =  &I^E_e (t,z) + 
 \frac{1}{z}\left(\sum\limits_{\{g| \on{age}(g) = 1\}} t^g I^E_g(t,z) \right) \\
  &+\frac{1}{z}\left( \sum \limits_{\{g_1, g_2| \on{age}(g_i) = 1\}} t^{g_1}t^{g_2} 
  I^E_{g_1, g_2}(t,z)  + \ldots \right).
 \end{split}
\end{equation}
For $g$ such that $\on{age}(g) \leq 1$ (including $g = e$), 
define the $A$ model hypergeometric functions 
\begin{equation}\label{e:IA}
I^A_g(t,z) = i^*\left(I^E_g(t, z)\right).
\end{equation}
\end{definition}

\begin{theorem}\label{t:A-model} 
Given $g = (r_0, \ldots, r_4)$ such that the age shift of $\cW_g$ is at most 
$1$, there exist functions $F_0(t)$, $G_0(t)$, and $H_g(t)$, 
determined explicitly by $I^E_g(t, z)$ 
such that $F_0$ and $H_g$ ($g \neq 0$) are invertible, and 
\begin{equation}\label{e:A-model}
 J_g^{\cW}(\tau(t),z) = 
 \frac{I^A_g(t,z)}{H_g(t)} \qquad \text{where } 
 \tau(t) = \frac{G_0(t)}{F_0(t)}. 
\end{equation}
\end{theorem}
\begin{remark} 
In the statement of the theorem, $F_0(t)$ and $G_0(t)$ do not depend on $g$, 
so the \emph{mirror map} $t \mapsto \tau(t) = G_0(t)/F_0(t)$ is well defined. 
\end{remark}

\subsection{Proof of Theorem~\ref{t:A-model}} \label{s:3.3}

There are two key ingredients in the proof.
The first one is the version of
\emph{quantum Lefschetz hyperplane theorem} (QLHT) for orbifolds
proved in \cite{CCIT}.
By Equation~\eqref{e:3.1.1}, $\cW$ is a hyperplane section of $\cY$
and hence $J^{\cW}$ can be calculated by QLHT.
Corollary~5.1 in \cite{CCIT} in particular implies the following:

\begin{theorem}[\cite{CCIT}]\label{t:QLHT} 
Let the setting be as above, with $E= \sO(5) \to \cY$. Then
\begin{equation} \label{e:3.3.1}
  I^{E}(\bt, z) = F(\bt) + \frac{G(\bt)}{z} + O(z^{-2}) 
\end{equation} 
for some $F$ and $G$ with $F$ scalar valued and invertible, and 
\begin{equation}\label{e:QLHT} 
  J^{\cY, \on{tw}}(\tau(\bt), z) = \frac{I^{E}(\bt, z)}
  {F(\bt)} \qquad \text{where } \tau(\bt) = \frac{G(\bt)}{F(\bt)}.
\end{equation}
\end{theorem}

The second ingredient is the explicit formula of $J_g^{\cY}$ 
from Section~\ref{s:2}.
Note that we are only concerned with those $g$ such that $i^* \ii_g \neq 0$
and $\on{age}(\ii_g) \leq 1$.
Therefore only those $J_g^{\cY}$ are listed.
The following is a straightforward corollary of 
Theorem~\ref{t:Zformula}, \eqref{e:2.2.2} and \eqref{e:2.2.3}
by equating the terms $Q^{c(d,h)} \ii_{h^{-1}} H^k$ of $Z_g$ with the terms
$q^d e^{dt} \ii_{h^{-1}} H^k$ of $J^\cY_g$.

\begin{corollary} \label{c:ambientJformula} 
The functions $J^{\cY}_g(t,z)$ are given by the following formulas.
\begin{enumerate} 
\item[(i)]  If $g = e = (0,0,0,0,0)$, 
\begin{equation}\label{e:jdim4}
 J_e^\cY = e^{t H/z}\left(1 + \sum_{\langle d\rangle = 0} 
 q^d e^{dt} \frac{1}{\prod \limits_{\substack{0 < b \leq d\\ \langle b \rangle = 0 }} 
   (bz - H)^5 }\right). 
\end{equation}
\item[(ii)] If $g = (0,0,0, r_1, r_2)$, 
let $g_1 = (-r_1,-r_1,-r_1,0, r_2-r_1) \, (\on{mod} 5)$ and let 
${g_2 = (-r_2,-r_2,-r_2,r_1-r_2,0) \, (\on{mod} 5)}$.  Then
\begin{align}
 J^Y_g = 
  &e^{t H/z}\ii_g \left(1 + \sum_{\langle d \rangle = 0}   \frac{q^d e^{dt}}
 {\prod \limits_{\substack{0 < b \leq d \\ \langle b \rangle = 0}} (H+bz)^3  \prod 
 \limits_{\substack{0 < b \leq d \\ \langle b \rangle = \left\langle \frac{r_2}{5} \right\rangle}}
 (H+bz)\prod 
 \limits_{\substack{0 < b \leq d \\ \langle b \rangle = \left\langle \frac{r_1}{5}\right\rangle}} 
 (H+bz)} \right)  \label{e:jdim1} \\
+ &e^{t H/z}\ii_{g_1}\left(
 \sum_{\langle d \rangle = \left\langle \frac{r_1}{5}\right\rangle} \frac{q^d e^{dt}}{\prod 
 \limits_{\substack{0 < b \leq d \\ \langle b \rangle = \left\langle \frac{r_1}{5} \right\rangle}}
 (H+bz)^3\prod 
 \limits_{\substack{0 < b \leq d \\ \langle b \rangle = 0}} (H+bz)\prod 
 \limits_{\substack{0 < b \leq d \\ \langle b \rangle = \left\langle \frac{2r_1}{5} \right\rangle}}
 (H+bz)}\right) \nonumber \\
+ &e^{t H/z}\ii_{g_2} \left(
 \sum_{\langle d \rangle = \left\langle \frac{r_2}{5}\right\rangle} \frac{q^d e^{dt}}{\prod 
 \limits_{\substack{0 < b \leq d \\ \langle b \rangle = \left\langle \frac{r_2}{5} \right\rangle}} 
 (H+bz)^3\prod 
 \limits_{\substack{0 < b \leq d \\ \langle b \rangle = \left\langle \frac{2r_2}{5} \right\rangle}} 
 (H+bz)\prod 
 \limits_{\substack{0 < b \leq d \\ \langle b \rangle = 0}} (H+bz)}\right). \nonumber
\end{align}
\item[(iii)] If $g = (0,0, r_1, r_1, r_2)$, let 
$g_1 = (-r_1, -r_1, 0,0,r_2 - r_1) \, (\on{mod} 5)$ and let 
$g_2 = (-r_2, -r_2, r_1 - r_2, r_1 - r_2, 0 ) \, (\on{mod} 5)$.  Then 
\begin{align}
  J^Y_g = 
  &e^{t H/z}\ii_g  \left( 1 + \sum_{\langle d \rangle = 0} \frac{q^d e^{dt}}
 {\prod \limits_{\substack{0 < b \leq d \\ \langle b \rangle = 0}} (H+bz)^2\prod 
 \limits_{\substack{0 < b \leq d \\ \langle b \rangle = \left\langle \frac{3r_2}{5} \right\rangle}} 
   (H+bz)^2\prod 
 \limits_{\substack{0 < b \leq d \\ \langle b \rangle = \left\langle \frac{2r_1}{5}\right\rangle}}
 (H+bz)}  \right)   \label{e:jdim0}  \\ 
  + &e^{t H/z}\ii_{g_1} \left(
 \sum_{\langle d \rangle = \left\langle \frac{r_1}{5}\right\rangle}  
 \frac{q^d e^{dt}}{\prod 
 \limits_{\substack{0 < b \leq d \\ \langle b \rangle = \left\langle \frac{r_1}{5} \right\rangle}} 
 (H+bz)^2\prod \limits_{\substack{0 < b \leq d \\ \langle b \rangle = 0}} 
 (H+bz)^2\prod 
 \limits_{\substack{0 < b \leq d \\ \langle b \rangle = \left\langle \frac{r_2}{5} \right\rangle}}
 (H+bz)}\right) \nonumber \\  
 + &e^{t H/z}\ii_{g_2} \left( 
 \sum_{\langle d \rangle = \left\langle \frac{r_2}{5}\right\rangle} 
 \frac{q^d e^{dt}}{\prod 
  \limits_{\substack{0 < b \leq d \\ \langle b \rangle = \left\langle \frac{r_2}{5} \right\rangle}}
 (H+bz)^2\prod 
 \limits_{\substack{0 < b \leq d \\ \langle b \rangle = \left\langle \frac{2r_1}{5} \right\rangle}}
 (H+bz)^2\prod \limits_{\substack{0 < b \leq d \\ \langle b \rangle = 0}} (H+bz)}\right).
 \nonumber
\end{align}
\end{enumerate}

In fact, due to the age requirement,
there are only two choices in case (ii) up to permutation:
$(r_1, r_2)=(2,3)$ or $(1,4)$.
In case (iii), only $(r_1, r_2)=(1,3)$ or $(2,1)$ are possible.
\end{corollary}

\begin{lemma} \label{l:3.8}
There are scalar valued functions $F_0(t), G_0(t)$ and 
$G_g(t)$ for each $g$ with $\on{age}(g) =1$, such that
\begin{equation*} 
  i^* \left( I^E(\bt, z) \right) = F_0(t) + \frac{G_0(t) H}{z} + 
  \sum_{\on{age}(g)=1} \frac{t^g G_g(t) \ii_g}{z} + R,
\end{equation*}
where $R$ denotes the \emph{remainder}, consisting of terms with
either the degrees in $t^g$'s greater or equal to $2$ or the degree
in $z^{-1}$ greater or equal to $2$.
In other words, if we write $G(\bt)$ from \eqref{e:3.3.1} as
\[
 G(\bt) = \ol{G_0}(\bt)H + \sum_g \ol{G_g}(\bt) \ii_g
\]
and denote $O(2)$ the terms with the degrees in $t^g$'s greater or equal to 
$2$, then
\[
 F(\bt) = F_0(t) + O(2), \quad \ol{G_0}(\bt) = G_0(t) + O(2), 
 \quad \ol{G_g}(\bt) = t^g G_g(t) + O(2).
\]
\end{lemma}

\begin{proof}
The proof of this lemma follows from Corollary~\ref{c:ambientJformula}
together with the following observations.
First, in case (ii) $i^*(\ii_{g_1})= i^*(\ii_{g_2}) =0$ due to dimensional
reasons.
Similarly with $i^*(\ii_{g_2})=0$ in case (iii).
Secondly, in case (iii) the $\ii_{g_1}$ term has higher $z^{-1}$ power:
The modification factor contributes terms of $z^{5d}$ plus lower order (in $z$)
terms. 
$i^* J^{\cY}_g$ contributes $z^{-(5d+1)}$ plus higher order (in $z^{-1}$) terms.
The combined contribution goes to the remainder $R$.
\end{proof}

With all this preparation, it is easy to prove Theorem~\ref{t:A-model}.

\begin{proof}[Proof of Theorem~\ref{t:A-model}]
Start by pulling back the equation \eqref{e:QLHT} to $\cW$.
Setting all $t^g =0$ we get \eqref{e:A-model} for the case $g=e$ 
if we let $H_e = F_0$:
\[
I^A_e(t) = i^*I^E_e (t) = i^* I^E(\bt) |_{\bt =tH}.
\]
Here by $\bt=tH$ we mean that setting all $t^g=0$ in \eqref{e:3.2.7}.
In the case $g \neq e$, take the partial derivative of \eqref{e:QLHT}
with respect to $t^g$ and then set all $t^g=0$.
Note that from \eqref{e:3.2.8}, we have
\[
 I^A_g(t) = i^*I^E_g (t) 
 = z \frac{\partial}{\partial t^g} i^*I^E(\bt) |_{\bt =tH}. 
\]
By Lemma~\ref{l:3.8} all the ``extra terms'' vanish and \eqref{e:A-model}
follows for $g \neq e$ after letting $H_g(t) = G_g(t)$.
The proof is now complete.
\end{proof}

\section{Periods and Picard--Fuchs equations} \label{s:4}

The theory of variation of Hodge structures (VHS) is closely related to the 
\emph{$B$ model} of a Calabi--Yau variety $X$, which encodes information 
about the deformations of complex structures on $X$.
By the local Torelli theorem for Calabi--Yau's, the Kodaira--Spencer spaces
inject to the tangent spaces of period domains and one can investigate the 
deformations of $X$ via VHS, which can be described by a 
system of 
flat connections on cohomology vector bundles.

For the benefit of the readers who come from the GWT side of mirror 
symmetry, we give a brief and self-contained summary of the parts of VHS theory
which are related to our work: the Gauss--Manin connection and
the associated notions of the period matrix and Picard--Fuchs equations.
For a more detailed introduction the reader may consult \cite{pG2},
\cite{pG1}. 

\subsection{Gauss--Manin connections, periods, and Picard--Fuchs equations}\label{s:4.1}
Over a smooth family of projective varieties $\pi: \mathscr{X} \to S$ 
of relative dimension $n$, 
we can consider the higher direct image sheaf (tensored with $\sO_S$) on $S$:
\[
 R^n \pi_*\CC \otimes \sO_S.
\]
The fiber over a point $t \in S$ of this sheaf is $H^n(X_t)$.  
This sheaf is locally free, and is naturally endowed with a \emph{flat} 
connection $\nabla^{GM}$, the \emph{Gauss--Manin} connection.
It can be defined in terms of the flat sections given by the lattice 
$R^n \pi_* \ZZ$ in $R^n\pi_*\CC \to S$, a \emph{local system}.
The Hodge filtration can be described fiberwise by 
\[
 \left(\sF^p\right)_t \cong \oplus_{a \geq p} H^{a, n-a}(X_t).
\]

We will be particularly interested in the case when the base $S$ is one
dimensional. 
Suppose now $S$ is an open curve 
and the family $\pi$ extends to a flat family over a proper curve $\bar{S}$. 
The vector bundle $R^n\pi_*\CC \otimes \sO_{S}$ extends to a vector 
bundle $\sH \to \bar{S}$ whose fiber over $t$ in $ S$ consists of the 
middle cohomology group $H^n(X_t)$.  
While it is not true that $\nabla^{GM}$ extends to a connection on all of $\sH$,
the singularities which arise are at worst a regular singularities \cite{pD}. 
This means that after choosing local coordinates, the connection matrix 
 acquires at worst a logarithmic pole at $t= 0$.  
Nevertheless we may still speak of flat (multi-valued) sections of 
$\nabla^{GM}$, controlled by the monodromy.

Let $\{ \gamma_i \}$ be a basis of $H_n(X_{t_0})$.
Since $\pi: \mathscr{X} \to S$ is smooth, it is a locally trivial fibration
and $n$-cycles $\gamma_i$ can be extended to \emph{locally constant} 
cycles $\gamma_i(t)$.
Let $\omega_t$ be a (local) section of $\sH$.
The functions $\int_{\gamma(t)} \omega_t$ are called the \emph{periods} and
by the local constancy of $\gamma(t)$
\[
 \frac{d}{d t}\left( \int_{\gamma(t)} \omega_t  \right) = 
 \int_{\gamma(t)}  \nabla^{GM}_t s(t).
\]  

The periods satisfy the \emph{Picard--Fuchs equations}, defined as follows.
Taking successive derivatives of $\omega_t$ with respect to the connection 
gives a sequence of sections 
\[
  \omega_t, \nabla^{GM}_t\omega_t, \ldots, \left(\nabla^{GM}_t\right)^k\omega_t, 
  \ldots .
\]  
Because the rank of $\sH$ is finite, for some $k$ there will exist a relation 
between these sections of the form 
\[ 
  \left(\nabla_t^{GM}\right)^k \omega_t + 
\sum_{i = 0}^{k-1} f_{i}(t)\left(\nabla_t^{GW}\right)^i \omega_t  = 0.
\]  
The corresponding differential equation 
\begin{equation}\label{e:PF} 
 \left( \left(\frac{d}{d t}\right)^k + 
 \sum_{i = 0}^{k-1} f_{i}(t)\left(\frac{d}{d t}\right)^i \right) \left( \int_{\gamma(t)} \omega_t \right)=0
\end{equation} 
is the Picard--Fuchs equation for $\omega_t$.  
The situation when the dimension of $S$ is greater than one is 
essentially the same, but \eqref{e:PF} is replaced by a PDE.

Let $\{\phi_i\}_{i \in I}$ be a basis of sections of $\sH$.  
Then if $\{\gamma_i\}_{i \in I}$ is a basis of locally
constant $n$-cycles, we can 
write the fundamental solution matrix of the Gauss-Manin
connection in coordinates as \[S = \left( s_{ij} \right) 
\text{ with } s_{ij} = \int_{\gamma_j} \phi_i.\] With this choice of 
basis, we see that the $i^{th}$ row of $S$ gives the periods for 
the section $\phi_i$.

\begin{remark}
In the literature, often (but not always) the term \emph{periods} are 
reserved for the case when $\phi(t)$ is a (holomorphic) $n$-form, 
i.e.\ a section of $\sF^n$, and
Picard--Fuchs equations only for periods in this restricted sense.
Here, we choose to use these terms in a more general sense defined above.
Note, however, by the results in \cite{BG}, for Calabi--Yau threefolds the 
general Picard--Fuchs equations can be determined from the restricted ones.
\end{remark}

\begin{remark} \label{r:4.2}
Let $U$ denote the Kuranishi space of the Calabi-Yau $n$-fold $X$.  
For the purpose of this paper, we use the term (genus zero part of)
\emph{$B$ model of $X$} to denote the vector bundle 
$\sH \to U$ with 
the natural (flat) fiberwise pairing and the Gauss--Manin connection.
\end{remark}

\subsection{Griffiths--Dwork method}  \label{s:4.2}

Let us assume now that the family $X_t$ is a family of hypersurfaces
defined by homogeneous polynomials $Q_t$ of  degree $d$ in $\PP^{n+1}$.
In this case the \emph{Griffiths--Dwork method} can be employed to
explicitly calculate the Picard--Fuchs equations.
We summarize the relevant results of \cite{pG1} here.
 
The method relies on Griffiths' work in \cite{pG1} showing that one can
calculate the period integrals on $X_t$ as one of \emph{rational} forms 
on $\PP^{n+1}$.
\emph{For the time being, let us fix $t$ and suppress it in the notation.}
Griffiths first shows that in fact any class $\Omega$ 
in $H^{n+1}(\PP^{n+1} \setminus X)$ can be represented in cohomology
by a \emph{rational} $n+1$ form.  In particular, let $\Omega_0$ be the
canonical $n+1$-form on $\PP^{n+1}$:
$\Omega_0= \sum_{i = 0}^{n+1} (-1)^{i} x_i dx_0 \cdots \hat{dx_i} \cdots dx_{n+1}$.
We can represent $\Omega$ by a rational form with poles in $X$,
\[
 \Omega = \frac{P(x)}{Q(x)^{k}} \Omega_0
\]
where $P(x)$ is a homogeneous polynomial with degree
$k d -(n+2)$.

The rational $n+1$ forms are then related to regular $n$ forms on $X$
via the residue map.
More precisely, let $A^n_k(X)$ denote the space of rational 
$(n+1)$-forms on $\PP^{n+1}$ with poles of order at most $k$ on $X$, and let 
\[
 \cH_k(X) := {A^{n+1}_k(X)}/{dA^{n}_{k-1}(X)}.
\] 
This gives an obvious filtration
\[
 \cH_1(X) \subset \cH_2(X) \subset \cdots \subset \cH_{n+1}(X) =: \cH(X).
\]
This description of rational forms interacts nicely with the Hodge filtration
$F^p$ of the \emph{primitive classes}.  
Griffiths proves that the following diagram
\begin{equation}\label{e:commdiagram}
 \begin{array}{ccccccc}
 \cH_1(X)  &\subset & \cH_2(X) &
 \subset  & \cdots  &\subset & \cH_{n+1}(X) \\
\quad \downarrow \Res & & \quad \downarrow \Res & & & & 
\quad \downarrow \Res \\
  F^{n}&  \subset & F^{n-1}&
 \subset  & \cdots & \subset & F^{0}   
 \end{array}
\end{equation} 
is commutative, and that each vertical arrow is surjective.
In particular, $\cH_{k+1}(X)/\cH_{k}(X) \cong F^{n-k}/ F^{n-k+1}.$

Now, for each $n$-cycle $\gamma$ in $H_n(X)$, let
\[ 
 T: H_n(X) \to H_{n+1}(\PP^{n+1} \setminus X)
\]
be the \emph{tube map} such that $T(\gamma)$ is a sufficiently small 
$S^1$-bundle around $\gamma$ in $\PP^{n+1} \setminus X$.
Griffiths then shows that the tube map is \emph{surjective} in general
and also \emph{injective when $n$ is odd}.
\begin{theorem} \label{t:4.2}
All primitive classes on $X$ can be represented as residues of 
rational forms on  $\PP^{n+1}$ with poles on $X$.
This representation is unique when $n$ is odd.
\end{theorem}
This follows from the surjectivity/injectivity of $\Res$ and $T$, 
as well as the residue formula
\[ 
  \frac{1}{2 \pi i} \int_{T(\gamma)} \Omega = \int_\gamma \Res(\Omega).
\]

Next Griffiths relates the rational forms to the Jacobian ring.
Let $J(Q) = \langle \partial Q/ \partial x_0, \ldots, \partial Q /
\partial x_{n+1} \rangle$  be the Jacobian ideal of $Q$.  
\begin{theorem} \label{t:4.3}
\begin{equation} \label{e:filtration}
 \CC[ x_0, \ldots, x_{n+1}]_{dk - n - 1}/J(Q) 
\cong F^{n - k}/ F^{n + 1 - k} \cong PH^{n - k, k}(V).
\end{equation}
\end{theorem}
The key relationship between rational forms is given by the following 
formula ((4.5) in \cite{pG1}) 
\begin{equation}\label{e:relate} 
\frac{\Omega_0}{Q(x)^{k}} 
\sum_{j = 0}^{n+1} B_j(x) \frac{\partial Q(x)}{\partial x_j} = 
\frac{1}{k-1}\frac{\Omega_0}{Q(x)^{k-1}} \sum_{j=0}^{n+1} 
\frac{\partial B_i(x)}{\partial x_j} + d \phi, 
\end{equation}  
where $\phi \in  A^n_{k-1}$.
Thus, the order of the pole of a form $\frac{P(x)}{{Q(x)}^{k}} \Omega_0$ 
can be lowered if and only if $P(x)$ is contained in $J(Q)$.
Thus by identifying the form 
$\Res \left(\frac{P(x)}{Q(x)^{k}} \Omega_0\right)$  with the homogeneous 
polynomial $P$, one obtains the isomorphism. 

The above results allow one to explicitly calculate the Picard--Fuchs
equations for certain families of forms $\omega_t$ on $X_t$.
As before, $X_t$ is a family of hypersurfaces defined by 
degree $d$ homogeneous polynomials $Q_t$.  
Then we can represent a family of forms as 
$\omega_t = \Res \left(\frac{P_t(x)}{Q_t(x)^{k}} \Omega_0\right)$.
Let $\gamma_t$ be a locally constant $n$ cycle as before, then
\[
 \begin{split}
  \frac{\partial}{\partial t} \int_{\gamma_t} \omega_t 
  =&\frac{\partial}{\partial t} \int_{\gamma_t}  
   \Res\left(\frac{P_t(x)}{Q_t(x)^{k}} \Omega_0\right)
  = \frac{\partial}{\partial t} \int_{T(\gamma_t)} 
   \frac{P_t(x)}{Q_t(x)^{k}} \Omega_0 \\
   = & \int_{T(\gamma_t)} \frac{\partial}{\partial t}
   \left(\frac{P_t(x)}{Q_t(x)^{k}}\Omega_0\right)
  =  \int_{\gamma_t} \Res\left(\frac{\partial}{\partial t}
   \left(\frac{P_t(x)}{Q_t(x)^{k}}\Omega_0\right)\right).
 \end{split}
\]
The third equality follows because a small change in $T(\gamma(t))$ will not change
its homology class.
In other words, letting $\nabla^{GM}$ denote the Gauss--Manin connection,  
\[
 \nabla^{GM}_t \Res\left(\frac{P_t(x)}{Q_t(x)^{k}}\Omega_0\right) = 
 \Res\left(\frac{\partial}{\partial t}
 \left(\frac{P_t(x)}{Q_t(x)^{k}}\Omega_0\right)\right),
\]  
allowing one to obtain the Picard--Fuchs equations of $\omega_t$
via explicit calculations of the polynomials (in the Jacobian rings).
An explicit example is given in the next section.

\section{$B$ model of the Fermat quintic $M$} \label{s:5}
We now turn to the specific case of the Fermat quintic threefold $M$ in $\PP^4$.
It has been shown that the Hodge diamonds of $M$ and $\cW$ are mirror symmetric
\[
 h^{p,q}(M) = h^{3-p, q}(\cW).
\]
In particular, the deformation family of $\cW$ is one-dimensional
while for $M$ the deformation is $101$ dimensional.

Recall in our study of the $A$ model of $\cW$, we restrict the Dubrovin connection
(i.e.~Frobenius structure) to to the ``small'' parameter $t$ 
corresponding to the hyperplane class $H$.
In the following discussions of the complex moduli of $M$, we will also study
the full period matrix for the Gauss--Manin connection, but
restricted to a particular deformation parameter.

Let 
\begin{equation} \label{e:5.1}
 Q_\psi(x) = x_0^5 + x_1^5 + x_2^5 + x_3^5 + x_4^5 - \psi x_0x_1x_2x_3x_4,
\end{equation}
and define the family $M_\psi = \{Q_\psi(x) = 0\} \subset \PP^4$.  When
writing the Picard-Fuchs equations it will
later become convenient to the coordinate change $t = -5\log(\psi)$.

\subsection{Picard--Fuchs equations for $M_{\psi}$} \label{s:5.1}

In the specific case of the family $M_\psi$, there is a 
``diagrammatic technique'',
pioneered in \cite{CDR} and refined in \cite{DGJ},
which utilizes the symmetry of $Q_\psi$ and $P$ to simplify the bookkeeping.

The starting point is the equation \eqref{e:relate}.
Consider the rational form 
\[ 
 \omega_\psi = \frac{P(x)}{Q_\psi(x)^{k}} \Omega_0, \quad
 P(x) = x_0^{r_0} \cdots x_4^{r_4}, \quad \text{with $\sum_{i=0}^4 r_i = 5(k - 1)$.}
\]
Fix $i$ between $0$ and $4$, and set $B_j = \delta_{ij} x_i P(x)$ for  
$0 \leq j \leq 4$.  
Noting that 
\[ 
 \frac{\partial }{ \partial x_i}Q_\psi(x) = 5 x_i^4 - 
 \psi x_0 \cdots \hat{x_j} \cdots x_4, 
\]
and applying \eqref{e:relate} with these choices of $B_j$
(and $k$ replaced by $k+1$), we arrive at 
\begin{equation} \label{e:5.1.1}
  5\int_{T(\gamma)}\frac{\left(x_i^5\right)P  }{Q_\psi^{k+1}}\Omega_0- 
  \psi\int_{T(\gamma)}\frac{\left( x_0 \ldots x_4\right)P }{Q_\psi^{k+1}}
  \Omega_0 
  = \frac{1 + r_i}{k}\int_{T(\gamma)} \frac{P}{Q_\psi^k} \Omega_0 
\end{equation} 
for any choice of cycle $\gamma \in H_n(X)$.
Note, however, that there is a degenerate case in the above setting:
in the case when $P(x)$ is independent of $x_i$, let $B_j = \delta_{ij} P(x)$.
Then in \eqref{e:relate}
we get
\begin{equation} \label{e:5.1.2}
  5\int_{T(\gamma)}\frac{\left(x_i^4 \right) P}{Q_\psi^{k+1}}\Omega_0- 
  \psi\int_{T(\gamma)} 
 \frac{\left( x_0 \ldots \hat{x_i} \ldots x_4\right)P }{Q_\psi^{k+1}} 
  \Omega_0  = 0.
\end{equation}
\emph{We can interpret this equation as allowing $r_i=-1$ in \eqref{e:5.1.1}}.

%

Furthermore, $\frac{\partial }{ \partial \psi}Q_\psi = -x_0 \cdots x_4$, 
and so we have the relationship 
\begin{equation} \label{e:der} 
 \frac{\partial}{\partial \psi} \int_{T(\gamma)} 
 \frac{P}{Q_\psi^{k}} \Omega_0 = k\int_{T(\gamma)} 
 \frac{\left( x_0 \cdots x_4\right)P }{Q_\psi^{k+1}}\Omega_0. 
\end{equation}

The authors in \cite{CDR, DGJ} apply \eqref{e:5.1.1} \eqref{e:5.1.2} 
and \eqref{e:der} recursively to get relations of the periods, 
hence the Picard--Fuchs equations.
For convenience of bookkeeping, one can keep track of the polynomial $P(x)$ 
by its exponents $(r_0, \ldots, r_4)$.
\eqref{e:5.1.1} can be understood symbolically as a relation between
$(r_0, \ldots, r_4)$, $(r_0,\ldots, r_i+5, \ldots, r_4)$ and 
$(r_0+1, \ldots, r_4 + 1)$.

Consider for example the case $P=1$ corresponding to $(0,\ldots,0)$.
Applying \eqref{e:der} four times, 
one may write the fourth derivative of $(0, \ldots, 0)$ as 
a multiple of $(4,\ldots,4)$. 
This may then be related to $(5, 5, 5, 5, 0)$ by
\eqref{e:5.1.2}.
Applying \eqref{e:5.1.1} to relate $(r_0, \ldots, r_4)$ to
a linear combination of $(r_0,\ldots, r_i-5, \ldots, r_4)$ and 
$(r_0+1, \ldots, r_i-4, \ldots, r_4 + 1)$ repeatedly, one can 
reduce to terms with $r_i \leq 4$ for all $i$.
In fact, eventually all terms will be of the form 
$\{ (r, r,\ldots, r) \}$ for $r =0, \ldots, 4$.
This can be seen by noting that \emph{none of \eqref{e:5.1.1}, \eqref{e:5.1.1} 
or \eqref{e:der} changes $r_i -r_j \, (\on{mod} 5)$.}
Hence, we have found a relation between the fourth derivative of 
$(0,\ldots,0)$ and $\{ (r,\ldots, r) \}$ for $r =0, \ldots, 4$.
By \eqref{e:der}, the various $(r,\ldots, r)$ are $r$-th derivatives of 
$(0,\ldots,0)$, and we obtain a fourth order ODE in $\psi$ 
for the period corresponding to $P=1$.
(See Table~1 below for the equation.)
Other cases can be computed similarly.
These arguments can be illuminated by diagrams in \cite{CDR, DGJ},
 hence the name \emph{diagrammatic technique}.

Now we apply this method to calculate the Picard--Fuchs equations
for the period integrals we are interested in.
For every $g = (r_0, \ldots, r_4) \in G$ (defined in Section~\ref{s:2.1}),
define
\[
 P_g(x) = x_0^{r_0} \cdots x_4^{r_4}
\]
and  
\[
 k = \left( \sum_{i = 0}^4 \frac{r_i}{5} \right) + 1 = \on{age}(g) +1.
\]
We will consider specific families of the form 
\begin{equation} \label{e:5.1.4}
  \omega_g (\psi)= \Res \left( \frac{\psi P_g(x)}{Q_\psi(x)^k} \Omega_0\right)
\end{equation}
For our purposes, \emph{it will be sufficient to 
consider families $\omega_g$ such that $P_g$ satisfies 
$\on{age}(g) \leq 1$ (i.e.\ $\sum_{i = 0}^4 r_i \leq 5$) 
and at least two of the $r_i$'s equal $0$.}
We observe that other $\omega_g$ can be obtained from differentiations
\eqref{e:der} or relations \eqref{e:5.1.1} and \eqref{e:5.1.2}
from the listed $\omega_g$.
For example, $(1,1, 1,1,1)$ is the derivative of $(0,0,0,0,0)$;
$(1,1,1,2,0)$ is related to $(0,0, 0, 1, 4)$ via
\[
  0 \equiv x_3 \partial_{x_4} Q_{\psi} = x_3 x_4^4 - \psi x_0 x_1 x_2 X_3^2 .
\]
We remark that these conditions on $g$ 
match the conditions on $A$ model computation 
in Section~\ref{s:3} perfectly.
In Claim~\ref{claim:6.6} it is shown that the derivatives of these
families generate all of $\sH$.

Table~1 below gives the Picard--Fuchs equation satisfied by each of 
the above-mentioned forms.  We label the forms by the corresponding 
5-tuple $g = (r_0, \ldots, r_4)$.  Note that permuting the $r_i$'s does not 
effect the differential equation, so we do not distinguish between 
permutations.  Here \[t  = -5 \log( \psi).\]
\begin{table}[htdp]
\begin{center}
\[
\begin{array}{|c|c|} \hline
\text{type}&\text{Picard--Fuchs equation}\\ \hline
\, &\, \\
(0,0,0,0,0)& (\frac{d}{dt})^4 - 5^5e^t(\frac{d}{dt} + \frac{1}{5})
(\frac{d}{dt} + \frac{2}{5})(\frac{d}{dt} + \frac{3}{5})(\frac{d}{dt} 
+ \frac{4}{5})\\ 
\, &\, \\
\hline
\, &\, \\
(0,0,0,1,4)& (\frac{d}{dt})^2 - 5^5e^t(\frac{d}{dt} + 2/5)(\frac{d}{dt} + 3/5)\\
\, &\, \\
\hline
\, &\, \\
(0,0,0,2,3)& (\frac{d}{dt})^2 - 5^5e^t(\frac{d}{dt} + 1/5)(\frac{d}{dt} + 4/5)\\
\, &\, \\
 \hline
\, &\, \\
(0,0,1,1,3)& (\frac{d}{dt})(\frac{d}{dt}- 1/5) - 5^5e^t(\frac{d}{dt} + 1/5)(\frac{d}{dt} + 3/5)\\ 
\, &\, \\
\hline
\, &\, \\
(0,0,2,2,1)& (\frac{d}{dt})(\frac{d}{dt}- 2/5) - 5^5e^t(\frac{d}{dt} + 1/5)(\frac{d}{dt} + 2/5)\\ 
\, &\, \\
\hline
\end{array}
\]
\vspace{10pt}
\caption{The Picard--Fuchs equations for forms $\omega_g$.}\label{table}
\end{center}
\end{table}
The same computation was done in \cite{CDR, DGJ}.
We note however that \emph{there are several differences} 
between the period integrals we consider, and those of \cite{DGJ}.  
First, our family $M_\psi$ differs from that in \cite{DGJ} by a factor of 5 in 
the first term.  
Second, the forms we consider \eqref{e:5.1.4} differ slightly from those 
considered in \cite{DGJ} by an extra factor of $\psi$ in the numerator (see 
remark~\ref{r:5.1}).  
Finally, our final equations use different coordinates than in \cite{DGJ}.  
However the same methods used in their paper can easily be modified to 
obtain the formulas we present here.

\begin{remark}\label{r:5.1}
The factor of $\psi$ in the numerator of 
\eqref{e:5.1.4} might appear unnatural at the first 
glance, but it can be considered as a way to change the form of the Picard-Fuchs equation, as
\[
 \frac{d}{d t} e^{-t/5} f(t) = 
 e^{-t/5} \left( - \frac{1}{5} +\frac{d}{d t} \right) f(t). 
\]
In the comparison of $A$ model and $B$ model
this modification will ensure that the $I$ functions from both sides coincide.
It is also used in the Mirror Theorem for the Fermat quintic.
\end{remark}

\subsection{$I^B$-functions} \label{s:5.2}
We can solve the above Picard-Fuchs equations with hypergeometric 
series. As in Section~\ref{s:2}, 
we will organize these solutions in the form of an $I$-function.  
For each of the above forms $\omega_g$, $I^B_g$ will be a function 
taking values in $H^*_{CR}(\cW) \cong H^*(I\cW)$, whose 
components give solutions to the corresponding Picard--Fuchs equation.
\begin{proposition}\label{p:5.1}  For the $g$ listed in table~\ref{table}, 
the components of $I^B_g(t,1)$
give a basis of solutions to the Picard--Fuchs equations 
for $\omega_g$, where $I^B_g(t,z)$ is given below.
\begin{enumerate}
\item[(i)]If $g = e = (0,0,0,0,0)$, 
\begin{equation}\label{e:idim3} 
I^B_e(t,z) = e^{t H/z}\left(1 + \sum_{\langle d \rangle = 0}  e^{dt} 
\frac{\prod \limits_{1 \leq m\leq5d}(5H + mz)}{\prod \limits_{\substack{0 < b \leq d \\ 
\langle b \rangle = 0}} (H + bz)^5}\right) 
\end{equation}  
\item[(ii)] If $g = (0,0,0,r_1, r_2)$, 
\begin{equation}\label{e:idim1}
 \begin{split}
 &I^B_g(t,z) = e^{t H/z}\ii_g \\ 
 &\Bigg(1 +  \sum_{\langle d \rangle = 0}   e^{dt} 
 \frac{\prod \limits_{1 \leq m\leq5d}(5H + mz)}{\prod 
 \limits_{\substack{0 < b \leq d \\ \langle b \rangle = 0}} (H + bz)^3\prod 
 \limits_{\substack{0 < b \leq d \\ \langle b \rangle = \left\langle \frac{r_2}{5}\right\rangle}} 
 (H + bz)\prod 
 \limits_{\substack{0 < b \leq d \\ \langle b \rangle = \left\langle \frac{r_1}{5}\right\rangle}} 
  (H + bz)}\Bigg) 
 \end{split}
\end{equation}
\item[(iii)] If $g = (0,0,r_1, r_1, r_2)$, let 
$g_1 = (-r_1, -r_1, 0,0, r_2 - r_1) (\on{mod} 5)$.  Then
\begin{equation} \label{e:idim0}
 \begin{split} 
  &I^B_g(t,z) = \\
  e^{t H/z}\ii_g &\left(1 + \sum_{\langle d \rangle = 0}  e^{dt} \frac{\prod 
  \limits_{1 \leq m\leq5d}(5H+ mz)}{\prod 
    \limits_{ \substack{0 < b \leq d \\ \langle b \rangle = 0}} (H+bz)^2\prod 
   \limits_{\substack{0 < b \leq d \\ \left\langle b \right\rangle = \left\langle \frac{3r_2}{5} \right\rangle}} 
  (H+bz)^2\prod 
 \limits_{\substack{0 < b \leq d \\ \langle b \rangle = \left\langle \frac{2r_1}{5}\right\rangle}}
 (H+bz)} \right) \\  
 +\:e^{t H/z}\ii_{g_1} &\left(
 \sum_{\langle d \rangle = \left\langle \frac{r_1}{5}\right\rangle}  e^{dt} 
 \frac{\prod \limits_{1 \leq m\leq5d}(5H+ mz)}{\prod 
 \limits_{\substack{0 < b \leq d \\ \langle b \rangle = \left\langle \frac{r_1}{5} \right\rangle}}
 (H+bz)^2\prod 
 \limits_{\substack{0 < b \leq d \\ \langle b \rangle = 0}} (H+bz)^2\prod 
 \limits_{\substack{0 < b \leq d \\ \langle b \rangle = \left\langle \frac{r_2}{5}\right\rangle}}
 (H+bz)}\right) 
 \end{split}
\end{equation}
\end{enumerate}
\end{proposition}

\begin{remark}  \label{r:5.3}
Note that the functions $I^B_g(t,z)$ in equations 
\eqref{e:idim3}, \eqref{e:idim1}, and \eqref{e:idim0}, 
are supported on spaces of dimension 3, 1, and 0 respectively.
So for each $g$, the number of components of $I^B_g(t,z)$ equals
the order of the corresponding Picard--Fuchs equation as desired.
\end{remark}

\section{Mirror Theorem for the mirror quintic: $A(\cW) \equiv B(M)$}
In this section, we will show the ``mirror dual'' version of 
(the mathematical version of) the \emph{mirror conjecture} 
by Candelas--de la Ossa--Greene--Parkes \cite{CDGP}. 
More specifically, we will show that the $A$ model of $\cW$ is
equivalent to the $B$ model of $M$, up to a mirror map.  

We start in~\ref{s:6.1} by stating
a ``classical'' mirror theorem relating 
the GWT of $\cW$ with the periods of $M_\psi$ 
on the level of generating functions.
This is exactly analogous to Givental's original formulation in \cite{aG1}.  
In~\ref{s:6.2} we give a brief explanation of how 
Givental's original mirror theorem
implies a full correspondence between the $A$ model of $M$ and 
the $B$ model of $\cW$.  Finally in~\ref{s:6.3}
we use similar methods as in~\ref{s:6.2} to prove a mirror theorem 
equating the $A$ model of $\cW$ to the $B$ model of $M$.

\subsection{A correspondence of generating functions} \label{s:6.1}
We will first show that the $I$-functions $I^A_g$ of the $A$ model of $\cW$
(Definition~\ref{d:Ifunction})
are identical to the $I$-functions $I^B_g$ of the $B$ model of $M_\psi$
defined in Section~\ref{s:5.2}.

\begin{remark}
Note that in the formula 
$I^A_g$, the Novikov
variable $q$ always appears next to $e^t$.  
There is therefore no harm in setting $q = 1$.  
We apply this
specialization in what follows.
\end{remark}

\begin{proposition}\label{p:6.1} 
Let $g = (r_0, \ldots, r_4) \in G$ satisfies the conditions 
$\on{age}(g) \leq 1$ and that at least two of $r_i$'s are equal to zero.
We have an $A$-interpretation of $g$ as parameterizing a component of $\cW_g$
in $I\cW$.
We have also a $B$-interpretation of $g$ in $\omega_g$ \eqref{e:5.1.4}
where $P_g$ denote the polynomial $x_0^{r_0}\cdots x_4^{r_4}$.
Then
\[
  I^A_g(t,z) = I^B_g(t,z).
\]
\end{proposition}

\begin{proof}
This follows from a direct comparison of formulas 
\eqref{e:jdim4}, \eqref{e:jdim1}, and \eqref{e:jdim0} from 
Corollary~\ref{c:ambientJformula} with formulas 
\eqref{e:idim3}, \eqref{e:idim1}, and \eqref{e:idim0} respectively.
\end{proof}

Combining Proposition~\ref{p:6.1} with Theorem~\ref{t:A-model}, we 
conclude that some periods from VHS of $M$ correspond to
the Gromov--Witten invariants of $\cW$.

\begin{corollary} \label{c:6.2}
For $g =(r_0, \ldots, r_4) \in G$ such that $\on{age}(g) \leq 1$ and  
$\cW_g$ is nonempty (i.e. at least two $r_i$'s vanish),
we have
\[
 J_g^{\cW}(\tau(t),z) = 
 \frac{I^B_g(t,z)}{H_g(t)} \qquad \text{where } 
 \tau(t) = \frac{G_0(t)}{F_0(t)}.   
\]
In other words, under the mirror map
\[ 
 t \mapsto \tau = \frac{G_0(t)}{F_0(t)},
\] 
the periods of $\frac{\omega_g}{H_g(t)}$ are equal to the coefficients of 
$J^\cW_g(\tau, 1)$.
\end{corollary}
This theorem should be viewed as an analogue of Givental's original
mirror theorem~\ref{t:msquintic} stated below.

\subsection{Mirror Theorem for the Fermat quintic revisited} 
\label{s:6.2}

To get some insight of the full correspondence, we return to the ``classical'' 
mirror theorem for the Fermat quintic threefold.
While this is not strictly necessary for the logical flow of the proof,
we feel that it illuminates our approach in a simpler setting.
We also strive to clarify certain points which are not entirely clear
in the literature.


Let $J^M(t, z)$ denote the small $J$-function for $M$ 
where $t$ is the coordinate of $H^2(M)$ dual to the hyperplane class $H$.  
Let $\cW_\psi$ denote
the one dimensional deformation family defined 
by the vanishing of $Q_\psi$ (see \eqref{e:5.1}) in $\cY$.  
\begin{equation} \label{e:Qt}
  \cW_\psi := \{ Q_\psi(x) =0 \} 
  \subset \cY.
\end{equation}  Let 
\[
\omega = \Res\left(\frac{\psi \Omega_0}{Q_\psi(x)}\right).
\]
As in section~\ref{s:5} there exists an $H^*(M)$-valued $I$-function, 
$I^B_{\cW_\psi}(t, z)$, such that the components of $I^B_{\cW_\psi}(t,1)$
give a basis of solutions for the Picard--Fuchs equations for 
$\omega_\psi$, where $t= -5 \log \psi$.

\begin{theorem}[Mirror Theorem \cite{aG1}\cite{LLY}] \label{t:msquintic} 
There exist explicitly determined functions $F(t)$ and $G(t)$, 
such that $F$ is invertible, and 
\[
  J^{M}(\tau(t),z) = \frac{I^B_{\cW_\psi}(t,z)}{F(t)} \qquad 
  \text{where } \tau(t) = \frac{G(t)}{F(t)}.
\]
%
%
%
\end{theorem}

We will show how Theorem~\ref{t:msquintic} implies a 
correspondence between the fundamental solution matrix of the Dubrovin 
connection for $M$ and that of the Gauss--Manin connection for $\cW_\psi$.
In order to emphasize the symmetry between the $A$ model and $B$ model,
we will denote the respective pairings as $(-,-)^A$ and $(-,-)^B$.
 
Let 
\[
 s = e^t = \psi^{-5}, 
\]
and consider the flat family $\cW_s$ over $S = \spec(\CC[s])$.  
In the Calabi--Yau case, 
the $H$ expansion of $I^B$ always occurs in the form of a function of $H/z$,
 in particular
$I^B_{\cW_s}$ is homogeneous of degree zero if one sets $\deg(z)=2$.
The same is true of $J^M$.
Thus, one may set $z=1$ without loss of information.
$I^B_{\cW_s}(t,1)$ gives a basis of solutions for the Picard--Fuchs equations 
of $\omega$.
In other words after an appropriate choice of basis
$\{s^B_0(t), \ldots, s^B_3(t)\}$ of solutions of $\nabla^{GM}$, 
\[ 
 ( s^B_i(t), \omega)^B = I^B_i(t,1),
\]
where $I^B_i(t,z)$ is the $H^i$ coefficient of $I^B_{\cW_s}(t,z)$.

By the same argument, if we choose an appropriate basis 
$\{s^A_0(\tau), \ldots, s^A_3(\tau)\}$ of solutions for 
$\nabla^{z}$, Section~\ref{s:1} shows that the coefficients $J^M_i(\tau,1)$ 
of the function $J^M(\tau,1)$ give us the functions 
\[ 
  ( s^A_i(\tau), 1 )^A = J^M_i(\tau,1) .
\]
 
Thus we can interpret Theorem~\ref{t:msquintic} as saying that after 
choosing correct bases of flat sections and applying the mirror map 
\[ 
  t \mapsto \tau = \frac{G(t)}{F(t)},
\] 
we have the equality 
\[ 
  ( s^B_i(t), \omega/F(t) )^B = \frac{I_i(t,1)}{F(t)} 
  = J_i(\tau, 1) = ( s^A_i(\tau), 1 )^A .
\]

To show the full correspondence between the
solution matrix for the Dubrovin connection for $M$ and that of the
Gauss--Manin connection on $S$, 
we must find a basis $\phi_0, \ldots , \phi_3$ of sections of 
$\sH$ and a basis $T_0, \ldots, T_3$ of sections of $H^{even}(M)$ such 
that for all $i$ and $j$,  
\begin{equation}\label{e:mx}
 ( s^B_i, \phi_j )^B = ( s^A_i, T_j )^A
\end{equation} 
As expected, we set $\phi_0 = \omega/F(t)$ and $T_0 = 1$. 

\begin{claim} \label{claim:6.4}
\[
  \phi_j = \left( \nabla^{GM}_{t} \right)^j\phi_0 \:\text{ for } \: 
   0 \leq j \leq 3
\] 
gives a basis of sections for $\sH$.
\end{claim}

\begin{proof}
This follows from standard Hodge theory for Calabi--Yau threefolds, 
but in this case can be explicitly calculated.
\begin{align} 
 \nabla^{GM}_{t} \phi_0 =& \frac{d}{d t}\left(\frac{1}{F(t)}\right)\omega + \frac{1}{F(t)} \nabla^{GM}_{t} \omega \nonumber\\ 
 =& -\frac{F'(t)}{F(t)}\phi_0 + \frac{1}{F(t)}\Res \left(\frac{d}{d t} \frac{\psi \Omega_0}{Q_\psi}\right) \nonumber\\ 
 =& -\frac{F'(t)}{F(t)}\phi_0 +\frac{1}{F(t)}\Res\left( s \frac{d}{d s} \frac{\psi \Omega_0}{Q_\psi}\right)\nonumber \\ 
 = & -\frac{F'(t)}{F(t)}\phi_0 +\frac{1}{F(t)}\Res\left( \frac{-\psi}{5} \frac{d}{d \psi} \frac{\psi\Omega_0}{Q_\psi}\right)\nonumber \\ 
 =& -\frac{F'(t)}{F(t)}\phi_0 + \frac{-\psi}{5 F(t)}\Res \left( \frac{\Omega_0}{Q_\psi} + \frac{x_0 \cdots x_4}{Q_\psi^2}\Omega_0 \right).\label{e:basis}
\end{align}  
Because of the last term in the above sum, the image of 
$\left(\nabla^{GM}_{t}\right) \phi_0$ in $\sF^2 /\sF^3$ is nonzero by 
\eqref{e:filtration}.  
Similarly, the image of 
$\left(\nabla^{GM}_{t}\right)^j \phi_0$ in $\sF^{3-j}/\sF^{3+1-j}$ for
 $1 \leq j \leq 3$ is nonzero, thus the sections $\phi_0, \ldots, \phi_3$ 
must be linearly independent.
\end{proof}

Note that 
\begin{align}\label{e:Abasis}
 &( s^B_i, \phi_1 )^B  = ( s^B_i, \nabla^{GM}_{t} \phi_0 )^B 
=\frac{\partial}{\partial t}  ( s^B_i, \phi_0 )^B =\\
& \frac{\partial}{\partial t} ( s^A_i, T_0 )^A 
  = \left(\frac{\partial \tau}{\partial t}\right) 
   \frac{\partial}{\partial \tau} ( s^A_i, T_0 )^A 
  = \left( s^A_i, \left(\frac{\partial \tau}{\partial t} \right) \nabla^{z}_\tau T_0 \right)^A .
 \nonumber
\end{align}
Therefore, if we set 
\[
 T_1 = \frac{\partial (G/F)}{\partial t} \nabla^{z}_\tau T_0,
\] 
we have the desired relationship 
\[
 ( s^B_i, \phi_1 )^B = ( s^A_i, T_1)^A.
\]  
If we similarly set 
\[
  T_k = \frac{\partial (G/F)}{\partial t} \nabla^{z}_\tau T_{k-1},
\] 
\eqref{e:mx} follows.  

This shows that the mirror map lifts to an isomorphism of vector bundles, and the connection is preserved.  
Indeed, the fundamental solution of the Gauss--Manin connection is a 
$4$ by $4$ matrix, where $4$ is the rank of $H^3(\cW)$.
On the other hand, the fundamental solution of the Dubrovin connection
is also a $4$ by $4$ matrix, where $4$ is the rank of $H^{even}(M)$.
We recall that the $J$-function can be thought of as the first row vectors 
of the fundamental solution matrix, as discussed in Section~\ref{s:1}.
The above discussion shows that we can extend the correspondence between
the first row of the fundamental solution to the full fundamental solution.

We summarize the above in the following theorem.

\begin{theorem} \label{t:MTfull}
The fundamental solutions of the Gauss--Manin connection for $\cW_s$
are equivalent, up to a mirror map, to the fundamental solutions of the 
Dubrovin connection for $M$, when restricted to $H^2(M)$.
\end{theorem}

\subsection{Mirror Theorem for the mirror quintic}\label{s:6.3}

In this subsection, we will extend the partial correspondence in 
Section~\ref{s:6.1} between the periods of $M_\psi$ and the $A$ model of $\cW$
to the full correspondence, generalizing the ideas in Section~\ref{s:6.2}.

Similar to the above, 
consider the flat family 
$M_s$ over $S = \spec(\CC[s])$ defined by \eqref{e:5.1}, 
where $s = e^t = \psi^{-5}$.
Corollary~\ref{c:6.2} states that some periods of $M_s$ 
correspond to Gromov--Witten invariants on $\cW$. 
We would like to extend this result to all periods.

First, we must choose a basis of sections of $\sH \to S$.  
Let $\omega_e$ denote the holomorphic family of (3,0)-forms corresponding
to $g= e=(0,\ldots,0)$ in~\eqref{e:5.1.4}. 
It is no longer true that derivatives of $\omega_e/F_0(t)$ with respect to the 
Gauss--Manin connection generate a basis of sections of $\sH$, 
thus it becomes necessary to consider the other forms $\omega_g$
satisfying the conditions formulated in Corollary~\ref{c:6.2}.  Namely, let
$\phi_e = \omega/F_0(t)$ and let $\phi_g = \omega_g/H_g(t)$ 
where $g$ satisfies $\on{age}(g) = 1$.  Consider the set of sections
\[ 
  \{\phi_0, \nabla^{GM}_t \phi_0, (\nabla^{GM}_t)^2 \phi_0,
  (\nabla^{GM}_t)^3 \phi_0\} \cup \{ \phi_g, \nabla^{GM}_t \phi_g\}. 
\] 

\begin{claim} \label{claim:6.6}
These forms comprise a basis of 
the Hodge bundle $\sH$. 
\end{claim}

\begin{proof}
The proof is similar to Claim~\ref{claim:6.4}.
We note that in the last four rows in Table~1, corresponding to age one type,
the dimensions are $20, 20, 30,$ and $30$.  Thus $|\{ \phi_g\}|=100$, and
there are exactly 204 forms in the above set.
One can check via \eqref{e:filtration} and another argument like in 
\eqref{e:basis} that these sections are in fact linearly independent.
\end{proof}

Then, as in~\eqref{e:Abasis} the periods of $(\nabla^{GM}_t)^k \phi_0$ 
correspond to the derivatives $\left(\frac{d}{d t}\right)^k J^{\cW}_e(\tau, 1)$,
and the periods of $\nabla^{GM}_t \phi_g$ correspond to 
$\left(\frac{d}{d t}\right) J^{\cW}_g(\tau, 1)$. 

Let $T_0 = 1$, and $T_k = \frac{\partial (G_0/F_0)}{\partial t} \nabla^{z}_\tau T_{k-1}$
for $0 \leq k \leq 3$.  Let $T_g = \ii_g$ and $T_g ' = \frac{\partial (G_0/F_0)}{\partial t} \nabla^{z}_\tau \ii_g$.  Then if we choose the correct basis of flat sections $\{s^B_i\}$
and $\{s^A_i\}$, we have that 
\[\begin{split}
 ( s^B_i,  (\nabla^{GM}_t)^k \phi_0 )^B &= ( s^A_i, T_k )^A,\\
 ( s^B_i,   \phi_g )^B &= ( s^A_i, T_g )^A\: \on{and}\\
 ( s^B_i,  \nabla^{GM}_t \phi_g )^B &= ( s^A_i, T_g ' )^A.
 \end{split}
\]
This implies that the set
\[ 
  \{T_0, T_1, T_2, T_3\} \cup \{ T_g, T_g '\}, 
\] 
is a basis of $TH^{even}_{CR}(\cW)$, and that with these  choices
of bases the solution matrices for the two respective connections are identical
after the mirror transformation.  Thus we obtain the full correspondence.

In terms of the language of Theorem~\ref{t:MTfull}, we can formulate
our final result in the following form.
On the side of the $A$ model of $\cW$, let $t$ be the dual coordinate of $H$;
on the side of $B$ model of $M_s$, let $t = \log(s)$. 
Then we have

\begin{theorem} \label{t:MMTfull}
The fundamental solutions of the Gauss--Manin 
connection $\nabla^{GM}_t$ for $M_s$ is
equivalent, up to a mirror map, to the fundamental solutions of
the Dubrovin connection $\nabla^z_t$ for $\cW$ restricted to $tH \in H^2(\cW)$.
\end{theorem}

\begin{remark}
Even though the base direction is constrained to one dimension instead
of the full $101$-dimension deformation space, our fundamental solutions
are full $204$ by $204$ matrices, as both ranks of $H^3(M)$ and
$H^{even}(\cW)$ are $204$.
\end{remark}

\end{document}